\documentclass[toc=bibnumbered,ngerman,english,titlepage=false]{scrartcl}

\usepackage[utf8]{inputenc}
\usepackage{babel}
\usepackage[T1]{fontenc}
\usepackage{lmodern}
\usepackage{microtype}
\usepackage{dsfont}

\usepackage{amsmath,amsthm,amssymb}
\usepackage{bm}
\usepackage{textcomp}
\usepackage{csquotes}
\usepackage{enumitem}
\usepackage[color links=true,allcolors=black]{hyperref}

\usepackage{graphicx}
\usepackage{float}
\usepackage{times}

\numberwithin{equation}{section} 

\newcommand{\N}{\ensuremath{\mathbb{N}}}
\newcommand{\Z}{\ensuremath{\mathbb{Z}}}
\newcommand{\R}{\ensuremath{\mathbb{R}}}
\newcommand{\T}{\ensuremath{\mathbb{T}}}

\newcommand{\C}{\ensuremath{\mathbb{C}}}

\newcommand{\eps}{\ensuremath{\varepsilon}}

\newcommand{\supp}{\textup{supp}}

\newcommand{\Ec}{\ensuremath{\mathcal{E}}}
\newcommand{\Fc}{\ensuremath{\mathcal{F}}}

\newcommand{\Lc}{\ensuremath{\mathcal{L}}}

\newcommand{\Oc}{\ensuremath{\mathcal{O}}}

\setkomafont{disposition}{\normalcolor\bfseries}
\setkomafont{section}{\LARGE} 
\setkomafont{subsection}{\Large} 
 
\begin{document}
\thispagestyle{plain}

\topmargin -18pt\headheight 12pt\headsep 25pt

\ifx\cs\documentclass \footheight 12pt \fi \footskip 30pt

\textheight 625pt\textwidth 431pt\columnsep 10pt\columnseprule 0pt

 \renewcommand{\headfont}{\slshape}      
 \renewcommand{\pnumfont}{\upshape}      
\setcounter{secnumdepth}{5}             
\setcounter{tocdepth}{5}             

\newtheorem{Definition}{Definition}[section]
\newtheorem{Satz}[Definition]{Satz}
\newtheorem{Lemma}[Definition]{Lemma}
\newtheorem{Korollar}[Definition]{Korollar}
\newtheorem{Corollary}[Definition]{Corollary}
\newtheorem{Bemerkung}[Definition]{Bemerkung}
\newtheorem{Remark}[Definition]{Remark}
\newtheorem{Proposition}[Definition]{Proposition}
\newtheorem{Beispiel}[Definition]{Beispiel}
\newtheorem{Theorem}[Definition]{Theorem}



\pdfbookmark[1]{Titlepage}{title}
\begin{center}
	{\Large Convergence of the Nonlocal Allen-Cahn Equation to Mean Curvature Flow\\[2ex]
	\textsc{Helmut Abels, Christoph Hurm, Maximilian Moser}\\[1ex]
}
\end{center}

\begin{abstract}
	\noindent\textbf{Abstract.} 
        We prove convergence of the nonlocal Allen-Cahn equation to mean curvature flow in the sharp interface limit, in the situation when the parameter corresponding to the kernel goes to zero fast enough with respect to the diffuse interface thickness. The analysis is done in the case of a $W^{1,1}$-kernel, under periodic boundary conditions and in both two and three space dimensions. We use the approximate solution and spectral estimate from the local case, and combine the latter with an $L^2$-estimate for the difference of the nonlocal operator and the negative Laplacian from Abels, Hurm \cite{AbelsHurm}. To this end, we prove a nonlocal Ehrling-type inequality to show uniform $H^3$-estimates for the nonlocal solutions.\\
        
        
\end{abstract}

\pdfbookmark[1]{Table of Contents}{toc}

\section{Introduction}\label{sec_intro}
Let $\T^n$, $n\in\{2,3\}$, be the $n$-dimensional torus and $T>0$ be fixed. We consider the following nonlocal Allen-Cahn equation,
\begin{alignat}{2}
    \partial_tc_\eps + \Lc_\eta c_\eps + \frac{1}{\eps^2}f^\prime(c_\eps) &= 0 &\quad & \text{ in }\T^n\times(0,T), \label{eq_ACnonloc_1}\\
    c_\eps|_{t=0} &=  c_{0,\eps}& \quad & \text{ in }\T^n. \label{eq_ACnonloc_2}
\end{alignat}
Here, $c_\eps: \T^n\times (0,T) \rightarrow\R$ is an order parameter representing the relative difference of the two phases and $f:\R\rightarrow\R$ is a double well potential with wells of equal depth and minima at $\pm 1$. Moreover, the parameter $\eps>0$ is related to the thickness of the diffuse interface separating the two phases. Finally, for $\eta>0$ the nonlocal operator is given by
\begin{align}\label{eq_nonlocal_operator}
		\Lc_\eta \varphi(x) := &\int_{\T^n}\tilde{J}_\eta(x-y)\left(\varphi(x)-\varphi(y)\right)\,dy \\
        = &\int_{\R^n}J_\eta(x-y)\left(\varphi(x)-\varphi(y)\right)\,dy \nonumber
        \quad\text{for all }x\in\T^n,
\end{align}
for integrable $\varphi:\T^n \rightarrow\R$ and suitable interaction kernels $J_\eta: \R^n \rightarrow\R$, cf. \ref{ass_2} below. Here, $\tilde{J}_\eta: \T^n\rightarrow\R$ denotes a $2\pi$-periodic extension of $J_\eta$.
In this setting, the corresponding bilinear form 
\begin{align}\label{eq_E_eta}
		\Ec_\eta(\varphi) := \frac{1}{4}\int_{\T^n}\int_{\T^n}\tilde{J}_\eta(x-y)\big|\varphi(x) - \varphi(y)\big|^2\,\text{d}y\,\text{d}x 
\end{align}
satisfies the following variational convergence
\begin{align}\label{Conv_BF}
    \Ec_\eta(\varphi) \rightarrow \frac{1}{2}\int_{\T^n}|\nabla\varphi|^2\;\text{d}x
\end{align}
as $\eta\rightarrow 0$ for all $\varphi\in H^1(\T^n)$. For details, we refer to the results by Ponce \cite{Ponce1, Ponce2}.

Based on \eqref{Conv_BF}, the nonlocal-to-local convergence $\mathcal{L}_\eta\varphi \rightarrow -\Delta\varphi$ as $\eta\rightarrow 0$ as well as corresponding results about the nonlocal-to-local convergence (without rates) of the Cahn-Hilliard equation have been shown in \cite{DavoliRanetbauerScarpaTrussardi,DavoliRoccaScarpaTrussardi,DavoliScarpaTrussardi,DavoliScarpaTrussardi1,Elbar}. Recently, the authors in \cite{AbelsHurm} established rates of convergence for the nonlocal operator $\mathcal{L}_\eta$, see Theorem \ref{th_AbelsHurm} below. This result has then been used in \cite{AbelsHurm} to prove the nonlocal-to-local limit with rates of convergence for the Cahn-Hilliard equation and the Allen-Cahn equation as well as in \cite{HKP, HM1} for coupled Cahn-Hilliard systems. 

In particular, solutions to the nonlocal Allen-Cahn equation converge to the solution to the corresponding local Allen-Cahn equation,
\begin{alignat}{2}
    \partial_tc_\eps -\Delta c_\eps + \frac{1}{\eps^2}f^\prime(c_\eps) &= 0 &\quad & \text{ in }\T^n\times(0,T), \label{eq_ACloc_1}\\
    c_\eps|_{t=0} &=  c_{\eps,0}& \quad & \text{ in }\T^n. \label{eq_ACloc_2}
\end{alignat}
It is well-known that in the sharp interface limit $\eps\rightarrow 0$ the Allen-Cahn equation \eqref{eq_ACloc_1}--\eqref{eq_ACloc_2} converges towards mean curvature flow.
This has first been analyzed in \cite{deMS} and \cite{Chen} using the method of matched asymptotic expansions as long as there is a smooth solution of the mean curvature flow equation. Moreover, Ilmanen \cite{Ilmanen} showed convergence globally in time to a Brakke solution of the mean curvature flow. On account of stability of the linearized Allen-Cahn operator, one may also derive rates of convergence. An alternative approach has been used by the authors in \cite{FLS}. Therein, the relative entropy method has been applied to prove the sharp interface asymptotics.  

We are interested in the sharp interface limit for the nonlocal Allen-Cahn equation, i.e.~we intend to analyse the limit $\eps\rightarrow 0$ for initial data close to a diffuse interface configuration. This is a very difficult task in general, therefore we restrict ourselves to the case when $\eta$ is very small compared $\eps$ (see Theorem \ref{th_conv} below for a quantification of the latter). The idea is to enforce closeness to the local situation where the sharp interface limit is well-known. Indeed, we use the approximate solution and the spectral estimate from the local case, cf.~Section \ref{sec_local} below, and combine the latter with an estimate for $\|\Lc_\eta c_\eps+\Delta c_\eps\|_{L^2(\T^n)}$ from Abels, Hurm \cite{AbelsHurm}, cf.~Theorem \ref{th_AbelsHurm} below. In order to close the Gronwall argument, it turns out that we need to control the $H^3$-norm of the nonlocal solutions uniformly. The latter is done in Section \ref{sec_nonlocAC_bound} with a novel nonlocal Ehrling inequality, cf.~Theorem \ref{th_nonloc_ehrling} below. 

To the best of our knowledge this is the first rigorous sharp interface limit result in the nonlocal case. The main result of this contribution is the following: 

\begin{Theorem}[\textbf{Convergence}]\label{th_conv}
    Let $n\in\{2,3\}$, $T_0>0$ and $(\Gamma_t)_{t\in[0,T_0]}$ with $\Gamma_t\subset(-\pi,\pi)^n$ for all $t\in[0,T_0]$ be a smoothly evolving compact closed hypersurface in $\T^n$ satisfying mean curvature flow, i.e.~$V_{\Gamma_t}=H_{\Gamma_t}$ for all $t\in[0,T_0]$, where $V_{\Gamma_t}$ is the normal velocity and $H_{\Gamma_t}$ the mean curvature of $\Gamma_t$. We use the usual notation for $\Gamma, \Omega^\pm$, the tubular neighbourhood $\Gamma(\delta)$, $\delta>0$ small, and for tangential and normal derivatives $\nabla_\tau$ and $\partial_n$, cf.~Section \ref{sec_local} below. Finally, let $M\in\N$. 

    Then there exists an $\eps_0>0$ and $c^A_\eps:\T^n\times[0,T_0]\rightarrow\R$ smooth for $\eps\in(0,\eps_0]$ such that $\lim_{\eps\rightarrow 0}c^A_\eps=\pm 1$ uniformly on compact subsets of $\Omega^\pm$ such that the following holds:
    \begin{enumerate}
        \item  Let $M\geq 3$ and consider initial values $c_{0,\eps}\in H^3(\T^n)$, $\eps\in(0,\eps_0]$ for \eqref{eq_ACnonloc_1}-\eqref{eq_ACnonloc_2} with 
        \begin{align*}
        \sup_{\eps\in(0,\eps_0]}\|c_{0,\eps}\|_{\infty}<\infty\quad\text{ and }\quad\|c_{0,\eps}-c^A_\eps|_{t=0}\|_{L^2(\T^n)}\leq R\eps^{M+\frac{1}{2}},\quad R>0\text{ fixed},\\
        \|c_{0,\eps}\|_{H^1(\T^n)}\leq\frac{C}{\eps^2},\quad \|c_{0,\eps}\|_{H^2(\T^n)}\leq \frac{C}{\eps^6}\quad\text{ and }\quad \|c_{0,\eps}\|_{H^3(\T^n)}\leq\frac{C}{\eps^{16}}.
        \end{align*}
        Then there are constants $c,\eps_1>0$ such that if $\eta=\eta(\eps)\leq cR\eps^{16+M+\frac{1}{2}}$, $\eps\in(0,\eps_1]$, then the solution of the nonlocal Allen-Cahn equation \eqref{eq_ACnonloc_1}-\eqref{eq_ACnonloc_2} from Theorem \ref{th_existence} satisfies for all $\eps\in(0,\eps_1]$ and $T\in(0,T_0]$, where $\overline{c}_\eps:=c_\eps-c^A_\eps$, 
        \begin{align*}
        \sup_{t\in[0,T_0]}\|\overline{c}_\eps(t)\|_{L^2(\T^n)} +\|\nabla\overline{c}_\eps\|_{L^2(\T^n\times(0,T)\setminus\Gamma(\delta))} &\leq C \eps^{M+\frac{1}{2}},\\
        \|\nabla_\tau\overline{c}_\eps\|_{L^2(\T^n\times(0,T)\cap\Gamma(\delta))} + \eps \|\partial_n\overline{c}_\eps\|_{L^2(\T^n\times(0,T)\cap\Gamma(\delta))} &\leq C\eps^{M+\frac{1}{2}}.
        \end{align*}
        \item Let $M=2$. Then the same statement holds when time $T$ is small. 
    \end{enumerate}
\end{Theorem}

\begin{Remark}\upshape
    The $c^A_\eps$ in Theorem \ref{th_conv} is the approximate solution from the local case, cf.~Theorem \ref{th_local_approx} below, and $M$ corresponds to the number of terms in the expansion.
\end{Remark}

The structure of this paper is as follows. In Section \ref{sec_prelim}, we introduce some notation, assumptions and preliminary results. Moreover, we prove a nonlocal version of the Ehrling inequality, which we need to derive uniform estimates of the solutions. In Section \ref{sec_nonlocAC_existence}, we then show existence, uniqueness, higher regularity and uniform $H^3$-bounds of solutions for the nonlocal Allen-Cahn equation \eqref{eq_ACnonloc_1}-\eqref{eq_ACnonloc_2}. Section \ref{sec_conv} contains the proof of Theorem \ref{th_conv} (Convergence). Therein, we show stability estimates, which then imply the convergence.

\section{Preliminaries}\label{sec_prelim}
\subsection{Notation}\label{sec_notation}
Let $n\in\N$, $1\leq p\leq \infty$ and $k\in\N$. For any open set $\Omega\subseteq\R^n$ or $\Omega=\T^n$ and a Banach space $X$, we write $L^p(\Omega,X)$ and $W^{k,p}(\Omega,X)$ for the Lebesgue and Sobolev spaces of functions on $\Omega$ with values in $X$. The corresponding norms are denoted by $\|.\|_{L^p(\Omega,X)}$ and $\|.\|_{W^{k,p}(\Omega,X)}$, respectively. We also write $H^k(\Omega,X):=W^{k,2}(\Omega,X)$. If $X=\R$, we leave out the $X$ in the notation. Moreover, $H^{-k}(\T^n)$ are the dual spaces corresponding to $H^k(\T^n)$.

We recall that functions $f\in L^1(\T^n,\C)$ can be characterised by their Fourier expansion $f(x) = \frac{1}{(2\pi)^n}\sum_{k\in\Z^n}\hat{f}_ke^{ik\cdot x}$ , where the Fourier coefficients are given by 
\[
\hat{f}_k = \hat{f}(k):=\int_{\T^n} e^{-i k\cdot x} f(x)\,dx\quad\text{ for }k\in\Z^n.
\]
Moreover, we recall the Convolution Theorem, i.e.~for functions $f,g\in L^1(\T^n,\C)$, their convolution $f\ast g:=\int_{\T^n} f(.-y)g(y)\,dy$ satisfies the identity $\widehat{f\ast g}=\hat{f}\hat{g}$. Let us also mention Plancherel's Theorem, i.e.~the map $\widehat{.}: L^2(\T^n,\C) \rightarrow \ell^2(\Z^n,\C)$ is an isomorphism. More precisely, it holds $\frac{1}{(2\pi)^n}\|(\hat{f}(k))_{k\in\Z^n}\|_{\ell^2(\Z^n,\C)}^2 = \|f\|_{L^2(\T^n,\C)}^2$ for all $f\in L^2(\T^n,\C)$. Finally, we recall the theorem of Riemann-Lebesgue, which states $\lim\limits_{|k|\rightarrow\infty}|\hat{f}(k)| = 0$ for all $f\in L^1(\T^n,\C)$. 

Moreover, $H^s_2(\T^n)$ for $s\in\R$ are the usual Bessel-Potential spaces endowed with the norm $\|f\|_{H^s_2(\T^n)}^2 := \sum_{k\in\Z^n}(1+|k|^2)^s|\hat{f}(k)|^2$. 
It is well-known that $H^k_2(\T^n)$ is isomorphic to $H^k(\T^n)$ for all $k\in\Z$.


For $f\in L^1(\R^n,\C)$, we define the Fourier transform of $f$ by $\Fc[f](\xi):=\int_{\R^n} e^{-i \xi\cdot x} f(x)\,dx$ for $\xi\in\R^n$. 
Note that for clarity we use a different notation as for the torus. 
Note that if $\supp f\subseteq [-\pi,\pi]^n$ and we roughly identify $\T^n\cong[-\pi,\pi]^n$, then $\hat{f}(k)=\Fc[f](k)$ for all $k\in\Z^n$. We recall that the Convolution Theorem as well as the Riemann-Lebesgue Theorem also hold true for the Fourier transform on $\R^n$. In this setting, Plancherel's Theorem states that $\Fc : L^2(\R^n,\C) \rightarrow L^2(\R^n,\C)$ is an isomorphism. 

Finally, we use the convention that constants can change from line to line, but are independent from the free variables, unless noted otherwise.

\subsection{Assumptions}\label{sec_assumptions}
\begin{enumerate}[label=\textnormal{(A.\arabic*)}]
    \item \label{ass_1} Let $\T^n$, $n\in\{2,3\}$, be the $n$-dimensional torus. 
    \item \label{ass_2}  
    Let $\eta>0$ and $J_\eta: \R^n\rightarrow[0,\infty)$ be a non-negative function given by $J_\eta(x) = \frac{\rho_\eta(|x|)}{|x|^2}$ for all $x\in\T^n$ and $J_\eta\in W^{1,1}(\R^n)$, where $(\rho_\eta)_{\eta>0}$ is a family of mollifiers satisfying
	\begin{align*}
		&\rho_\eta: \R\rightarrow [0,\infty),\quad\rho_\eta\in L^1(\R),\quad\rho_\eta(r)=\rho_\eta(-r)\quad\text{for all }r\in\R, \,\eta>0, \\
        &\rho_\eta(r)=\eta^{-n} \rho_1\left(\frac{r}{\eta}\right)\quad\text{ for all }r\in\R,\\
		&\int_{0}^\infty\rho_\eta(r)\:r^{n-1}\,dr = \frac{2}{C_n}\quad\text{for all }\eta>0, \\
		&\lim\limits_{\eta\searrow 0}\int_{\delta}^\infty\rho_\eta(r)\:r^{n-1}\,dr = 0\quad\text{for all }\delta>0,
	\end{align*}
	where $C_n := \int_{\mathbb{S}^{n-1}}|e_1\cdot\sigma|^2\,d\mathcal{H}^{n-1}(\sigma)$. Moreover, let $\rho_1$ be compactly supported in $(-\pi,\pi)$. Finally, let $\tilde{J}_\eta: \T^n\rightarrow[0,\infty)$ be the $2\pi$-periodic extension of $J_\eta$.
    \item \label{ass_3} $f:\R\rightarrow\R$ is a smooth double-well potential with wells of equal depth, more precisely,
    \[
    f\in C^\infty(\R),\quad f(\pm 1)=f'(\pm 1)=0, \quad f''(\pm 1)>0,\quad f>0\text{ in }(-1,1)
    \]
    and $f'<0$ in $(-\infty,-R_0)$ as well as $f'>0$ in $(R_0,\infty)$ for some $R_0\geq1$. Finally, we assume
    \[
    |f'(r)|\leq C(1+|r|^3)\quad\text{ for all }r\in\R
    \]
    and $f''\geq -\alpha$ for some $\alpha\geq 0$.
 \end{enumerate}

 \begin{Remark}\upshape
     Note that a typical example for $J_\eta$ is given through the choice 
     \[
     \rho_1(r):=C_{\beta,n,\tilde{\rho}}\,|r|^{\beta} \tilde{\rho}(r)\quad\text{ for all }r\in\R,
     \]
     where $\beta\in(3-n,2)$ to ensure the integrability conditions and $\tilde{\rho}\in C_0^\infty((-\pi,\pi))$, $\tilde{\rho}\not\equiv 0$ is symmetric, non-negative and $C_{\beta,n,\tilde{\rho}}>0$ is a normalization constant. Note that this choice corresponds to the $W^{1,1}$-kernel setting in Davoli, Scarpa, Trussardi \cite{DavoliScarpaTrussardi}.
 \end{Remark}

 \begin{Remark}\label{th_opt_profile}\upshape
     Note that the assumptions on the double-well potential $f$ in \ref{ass_3} are standard. Under these assumptions, we then define the function $\theta_0\in C^2(\R)$ to be the unique solution to the ordinary differential equation,
    \begin{align}\label{eq_opt_profile}
    -\theta_0''+ f'(\theta_0) = 0 \quad\text{ in }\R, \quad \lim_{\rho\rightarrow\pm\infty} \theta_0(\rho) = \pm 1, \quad \theta_0(0)=0.
    \end{align}
    The function $\theta_0$ is also referred to as the \textit{optimal profile}. 
 \end{Remark}
 
\subsection{Inequalities}\label{sec_inequalities}

\begin{Theorem}\label{th_AbelsHurm}
		Let $n\in\{2,3\}$. Moreover, let $\Lc_\eta$ be defined as in \eqref{eq_nonlocal_operator} and $J_\eta$ satisfy \ref{ass_2} from Section \ref{sec_assumptions}. Then for all $c\in H^3(\T^n)$ it holds for a constant $K>0$ independent of $\eta$
		\begin{align}\label{eq_AbelsHurm}
			\Big\|\Lc_\eta c + \Delta c\Big\|_{L^2(\T^n)}\leq K\eta\|c\|_{H^3(\T^n)}. 
		\end{align}
	\end{Theorem} 
 \begin{proof}
     The proof of this theorem can be done analogously as in \cite[Lemma 3.1]{AbelsHurm} by using Fourier series instead of Fourier transformation.
 \end{proof}
\begin{Theorem}[\textbf{Nonlocal Ehrling Inequality}]\label{th_nonloc_ehrling}
    Let $n\in\N$, $R\geq 1$, $0<\eta\leq\frac{1}{R}$ and $\Ec_\eta$ be defined as in \eqref{eq_E_eta} with $J_\eta$ satisfying \ref{ass_2} from Section \ref{sec_assumptions}. Then for all $u\in L^2(\T^n)$ it holds with $C>0$ independent of $R,\eta,u$,
    \[
    \|u\|_{L^2(\T^n)}^2 \leq \frac{C}{R^2} \Ec_\eta(u) + C R^2 \|u\|_{H^{-1}(\T^n)}^2. 
    \]
\end{Theorem}
\begin{proof}
    Let $R\geq 1$ and $0<\eta\leq\frac{1}{R}$ be arbitrary. Using Plancherel's theorem, the convolution theorem and the definition of $J_\eta$, we obtain
    \begin{align*}
        \Ec_\eta(u) = \int_{\T^n}\Lc_\eta u(x)u(x)\,\text{d}x 
        =\frac{1}{(2\pi)^n}\sum_{k\in\Z^n}\widehat{\Lc_\eta u}(k)\overline{\widehat{u}(k)} 
        =\frac{1}{(2\pi)^n}\sum_{k\in\Z^n}\left(\widehat{J_\eta}(0) - \widehat{J_\eta}(k)\right)|\hat{u}(k)|^2
    \end{align*}
    for all $u\in L^2(\T^n)$. 
    Observe that the interaction kernel $J_\eta$ is compactly supported in $(-\pi,\pi)^n$.
    Therefore,
    \begin{align}\label{eq_Fourier}
        \widehat{J_\eta}(k) = \int_{\T^n}\tilde{J}_\eta(z)e^{-ik\cdot z}\,\text{d}z = \int_{\R^n}J_\eta(z)e^{-ik\cdot z}\,\text{d}z = \Fc[J_\eta](k)
    \end{align}
    for all $k\in\Z^n$.
    For the following, we define the auxiliary function 
    \begin{align*}
    \Psi:\R^n\rightarrow\C:\xi\mapsto  \frac{1}{(2\pi)^n}\left(\widehat{J_1}(0)-\Fc[J_1](\xi)\right).
    \end{align*}
    Due to \eqref{eq_Fourier} and since $\Fc[J_1]$ is a function on $\R^n$, it follows by standard Fourier analysis that $\Psi\in C^\infty(\R^n)$ and $\Psi$ is $\R$-valued because $J_1$ is even.
    Moreover, a scaling argument yields
    \begin{align*}
    \frac{1}{(2\pi)^n}\left(\widehat{J_\eta}(0) - \widehat{J_\eta}(k)\right)= \frac{1}{\eta^2} \Psi(k\eta)\quad\text{ for all }k\in\Z^n,\;0 < \eta \leq \frac{1}{R}.
    \end{align*}
    We need to derive more properties of $\Psi$. It holds $\Psi\geq 0$ because of
    \begin{align*}
        \big|\Fc[J_1](\xi)\big| \leq \int_{\R^n}J(x)\,\text{d}x = \widehat{J_1}(0)\quad\text{ for all }\xi\in\R^n.
    \end{align*}
    Moreover, a contradiction argument shows $\Psi(\xi)>0$ for all $\xi\in\R^n\setminus\{0\}$.
    
    For the gradient we obtain
    \begin{align*}
        \nabla\Psi(\xi) = -i\int_{\R^n}J_1(x)e^{-ix\cdot\xi}x\,\text{d}x.
    \end{align*}
    In particular, this implies $\nabla\Psi(0)=0$ since the interaction kernel $J_1$ is even. Moreover, we have
    \begin{align*}
        D^2\Psi(\xi)_{k,j} = \int_{\R^n}J_1(x)e^{-ix\cdot\xi}x_kx_j\,\text{d}x,
    \end{align*}
    and therefore $D^2\Psi(0) = \textup{Id}$ due to the properties of $J_1$. Hence, a Taylor expansion around $0$ shows that $\Psi$ has quadratic growth in a small neighbourhood around $0$. In particular, a compactness argument yields
    \begin{align*}
    \Psi(\xi)\geq c_0|\xi|^2\quad\text{ for all }|\xi|\leq 1,
    \end{align*}
    where $c_0>0$ is a constant independent of $R,\eta,u$.
    Finally, because of the Riemann-Lebesgue Lemma we have $|\widehat{J_1}(\xi)| \rightarrow 0$ for $|\xi|\rightarrow\infty$ which implies that $\Psi(\xi)\rightarrow \frac{1}{(2\pi)^n}\widehat{J_1}(0)$ for $|\xi|\rightarrow\infty$. In particular, by continuity it holds
    \begin{align*}
        \Psi(\xi) \geq c_1\quad\text{ for all }|\xi|\geq 1,
    \end{align*}
    where $c_1>0$ is a constant independent of $R,\eta,u$. 
    
    Finally, we can prove the main assertion. Due to the arguments in the steps before, we have 
    \begin{align*}
        \Ec_\eta(u) &= \sum_{k\in\Z^n}\frac{\Psi(k\eta)}{\eta^2}|\hat{u}(k)|^2 \\
        &\geq \sum_{k\in\Z^n,|k|\leq R}c_0|k|^2|\hat{u}(k)|^2 
        +\sum_{k\in\Z^n,|k|\geq \frac{1}{\eta}}\frac{c_1}{\eta^2}|\hat{u}(k)|^2 
        +\sum_{k\in\Z^n,\frac{1}{\eta}> |k| > R}c_0|k|^2|\hat{u}(k)|^2 \\
        &\geq \sum_{k\in\Z^n,|k|\leq R}c_0|k|^2|\hat{u}(k)|^2 
        +c\sum_{k\in\Z^n,|k|> R}R^2|\hat{u}(k)|^2 \\
        &\geq c\sum_{k\in\Z^n,|k|> R}R^2|\hat{u}(k)|^2=c\sum_{k\in\Z^n}R^2|\hat{u}(k)|^2 
        - c\sum_{k\in\Z^n,|k|\leq R}R^2|\hat{u}(k)|^2,
    \end{align*}
    where $c:=\min\{c_0,c_1\}>0$. Noting that $R^2 \leq \frac{2R^4}{1+|k|^2}$ for all $k\in\Z^n$ with $|k|\leq R$, we obtain
    \begin{align*}
        \Ec_\eta(u) &\geq c\sum_{k\in\Z^n}R^2|\hat{u}(k)|^2 -c\sum_{k\in\Z^n}\frac{2R^4}{1+|k|^2}|\hat{u}(k)|^2 \\
        &= (2\pi)^n cR^2\|u\|_{L^2(\T^n)}^2 - 2cR^4\|u\|_{H^{-1}_2(\T^n)}^2,
    \end{align*}
    and thus the assertion follows.
\end{proof}

\subsection{Results for the Local Allen-Cahn Equation}\label{sec_local}
First, we fix some notation concerning the evolving hypersurface. 

Let $T_0>0$ and $(\Gamma_t)_{t\in[0,T_0]}$ with $\Gamma_t\subset(-\pi,\pi)^n$ for all $t\in[0,T_0]$ be a smoothly evolving compact hypersurface in $\T^n$. Moreover, let $\vec{n}_{\Gamma_t}=\vec{n}(.,t)$ be the exterior normal of $\Gamma_t$ for all $t\in[0,T_0]$ and $\Gamma:=\bigcup_{t\in[0,T_0]}\Gamma_t\times\{t\}$. Here $\Gamma_t$ separates $\T^n$ into two disjoint connected domains $\Omega^\pm_t$ for all $t\in[0,T_0]$, where $\vec{n}(.,t)$ points into $\Omega^\pm_t$, and we set $\Omega^\pm:=\bigcup_{t\in[0,T_0]}\Omega^\pm_t\times\{t\}$. We consider a reference hypersurface $\Sigma$ and a smooth parametrization $X_0:\Sigma\times[0,T_0]\rightarrow\T^n$ such that $\overline{X}_0:=(X_0,\textup{pr}_t):\Sigma\times[0,T_0]\rightarrow\Gamma$ is a smooth diffeomorphism. Then we denote by
\[
X:[-\delta,\delta]\times\Sigma\times[0,T_0]\rightarrow\T^n: (r,s,t)\mapsto X_0(s,t) + r \vec{n}(s,t)
\]
for $\delta>0$ small the standard tubular neighbourhood coordinate system. 

It is well known that, for $\delta_0>0$ small and all $\delta\in(0,\delta_0]$, $\overline{X}:=(X,\textup{pr}_t)$ is a smooth diffeomorphism onto the neighbourhood $\overline{\Gamma(\delta)}$ of $\Gamma$, where $\Gamma(\delta):=\bigcup_{t\in[0,T_0]}\Gamma_t(\delta)\times\{t\}$ and $\Gamma_t(\delta):=\{x\in\T^n:\textup{dist}(x,\Gamma_t)<\delta\}$ is the tubular neighbourhood of $\Gamma_t$ for $t\in[0,T_0]$. 

Finally, we set $(r,s,\textup{pr}_t):=\overline{X}^{-1}:\overline{\Gamma(\delta)}\rightarrow[-\delta,\delta]\times\Sigma\times[0,T_0]$ and define the tangential and normal derivative as
\[
\nabla_\tau\psi := (D_xs)^\top \left[\nabla_\Sigma(\psi\circ\overline{X})\circ\overline{X}^{-1}\right]\quad\text{ and }\quad \partial_n\psi:=\partial_r(\psi\circ\overline{X})\circ\overline{X}^{-1}
\]
for sufficiently smooth $\psi:\Gamma(\delta)\rightarrow\R$, $\delta\in(0,\delta_0]$. Note that $|\nabla\psi|^2=|\nabla_\tau\psi|^2+|\partial_n\psi|^2$.

In the following theorem, we recall a standard result from asymptotic expansions for the local Allen-Cahn equation \eqref{eq_ACloc_1}.
\begin{Theorem}[\textbf{Approximate Solution for Local Allen-Cahn Equation}]\label{th_local_approx}
    Let the notation be as above, $\Gamma$ be evolving according to mean curvature flow and let $M\in\N\cup\{0\}$.
    
    There are smooth $h_j:\Sigma\times[0,T_0]\rightarrow\R$, 
    $u_{j+1}:\R\times\Sigma\times[0,T_0]\rightarrow\R$ for $j=1,...,M$ with
    \[
    \partial_\rho^j\partial_s^k\partial_t^l u_{j+1}(\rho,s,t)= \Oc(e^{-\beta|\rho|}) \quad\text{ uniformly in }(\rho,s,t),
    \]
    such that the following holds: defining for $\eps>0$, with the optimal profile $\theta_0$ as in Remark \ref{th_opt_profile},
    \begin{align*}
    \rho_\eps(x,t)&:=\frac{r(x,t)-\eps h_\eps(s(x,t),t)}{\eps},\quad h_\eps:=\sum_{j=1}^M \eps^{j-1}h_j,\\
    u^I_\eps(x,t) &:=\theta_0(\rho_\eps(x,t))+ \sum_{k=2}^{M+1}\eps^j u_{j+1}(\rho_\eps(x,t),s(x,t),t)\quad\text{ in }\Gamma(\delta)
    \end{align*}
    as well as with a cutoff $\eta:\R\rightarrow[0,1]$ smooth with $\eta(r)=1$ for 
    $|r|\leq 1$, $\eta(r)=0$ for $|r|\geq 2$,
    \begin{align*}
        u^A_\eps := \begin{cases}
            \eta(\frac{r}{\delta_0}) u^I_\eps + (1-\eta(\frac{r}{\delta_0})) \textup{sign}(r)\quad &\text{ in }\Gamma(\delta_0),\\
            \pm 1 \quad &\text{ in }\T^n\times(0,T_0)\setminus\Gamma(\delta_0),
        \end{cases}
    \end{align*}
    then $u^A_\eps$ is smooth, uniformly bounded for small $\eps$ and for $r^A_\eps:=\partial_tu^A_\eps-\Delta u^A_\eps +\frac{1}{\eps^2} f'(u^A_\eps)$, the remainder of $u^A_\eps$ in the local Allen-Cahn equation \eqref{eq_ACloc_1}, it holds
    \begin{alignat*}{2}
    \left|r^A_\eps\right|&\leq C(\eps^M e^{-c|\rho_\eps|} + \eps^{M+1})&\quad&\text{ in }\Gamma(\delta_0),\\
    \quad r^A_\eps&=0&\quad&\text{ in }\T^n\times(0,T_0)\setminus\Gamma(\delta_0),
    \end{alignat*}
    where $c,C>0$ are some constants and $\eps>0$ is small.
\end{Theorem}
\begin{proof}
    This essentially follows from the computations in Chen, Hilhorst, Logak \cite{CHL} and Abels, Liu \cite{ALiu}, see also the expansions in Moser \cite{MoserACvACND}.
\end{proof}

Finally, we recall a classical spectral estimate for the linearized local Allen-Cahn operator. 

\begin{Theorem}[\textbf{Spectral Estimate for Local Allen-Cahn Operator}]\label{th_local_spectral}
Let $\delta\in(0,\frac{\delta_0}{2}]$ and $h_\eps:\Sigma\times[0,T_0]\rightarrow\R$ be such that $h_\eps(.,t)$ is continuous for all $t\in[0,T_0]$ and $\|h_\eps\|_\infty\leq C_0$, where $\eps>0$ is small and $C_0>0$ is a fixed constant. Moreover, we define the scaled variable
\[
\rho_\eps:=\frac{r-\eps h_\eps(s,t)}{\eps} \quad \text{ in }\overline{\Gamma(2\delta)}.
\]
For $\eps>0$ small we consider $u^A_\eps:\T^n\times[0,T_0]\rightarrow\R$ measurable with the property
\begin{align*}
    u_\eps^A=\begin{cases}
        \theta_0(\rho_\eps) + \Oc(\eps^2) & \text{ in }\Gamma(\delta),\\
        \pm 1 + \Oc(\eps) & \text{ in }[\T^n\times(0,T_0)]\setminus\Gamma(\delta),
    \end{cases}
\end{align*}
where $\theta_0$ is the optimal profile from \eqref{eq_opt_profile}.

    Then there exist $\eps_0,\overline{C}>0$ independent of $h_\eps$ for fixed $C_0$ such that for all $\eps\in(0,\eps_0]$, $t\in[0,T_0]$ and all $\psi\in H^1(\T^n)$ it holds
    \[
    \int_{\T^n}|\nabla\psi|^2+ \frac{1}{\eps^2} f''(u_\eps^A(.,t))\psi^2\,dx \geq -\overline{C}\|\psi\|_{L^2(\T^n)}^2 + \|\nabla\psi\|_{L^2(\T^n\setminus\Gamma_t(\delta))}^2 + \|\nabla_\tau\psi\|_{L^2(\Gamma_t(\delta))}^2.
    \]
\end{Theorem}
\begin{proof}
    This goes back to Chen \cite{ChenSpectrums} and Abels, Liu \cite[Theorem 2.13]{ALiu}.
\end{proof}

\section{Properties of Solutions to the Nonlocal Allen-Cahn Equation}\label{sec_AC}
Let the assumptions \ref{ass_1}-\ref{ass_3} hold, $\eps>0$ and $\mathcal{L}_\eta$ be given as in \eqref{eq_nonlocal_operator}. In this section we consider the nonlocal Allen-Cahn equation \eqref{eq_ACnonloc_1}-\eqref{eq_ACnonloc_2} and show existence, uniqueness, higher regularity and boundedness of solutions in Section \ref{sec_nonlocAC_existence} and uniform $H^3$-bounds for solutions for sufficiently small $\eps$ and $\eta\leq\eta(\eps)$ in Section \ref{sec_nonlocAC_bound}.

\subsection{Existence, Uniqueness and Maximum Principle}\label{sec_nonlocAC_existence}

\begin{Lemma}[\textbf{Uniqueness}]\label{th_uniqueness}
    Let the assumptions \ref{ass_1}-\ref{ass_3} hold, $T>0$ and $\eps>0$. For initial values in $H^1(\T^n)$, solutions of \eqref{eq_ACnonloc_1}-\eqref{eq_ACnonloc_2} with the regularity $H^1(0,T;L^2(\T^n)) \cap L^\infty(0,T;H^1(\T^n))$ are unique.
\end{Lemma}
\begin{proof}
Let $c_0\in H^1(\T^n)$ and $c_1, c_2\in H^1(0,T;L^2(\T^n))\cap L^\infty(0,T;H^1(\T^n))$ be solutions of 
\begin{alignat*}{2}
\partial_tc+\Lc_\eta c+f'(c)&=0 &\quad & \text{ in }\T^n\times(0,T),\\
c|_{t=0} &= c_0 &\quad &\text{ in }\T^n.
\end{alignat*}
Then, $\tilde{c} := c_1-c_2$ solves 
\begin{alignat*}{2}
    \partial_t\tilde{c}+\Lc_\eta\tilde{c} + f'(c_1) - f'(c_2) &=0 &\quad&\text{ in }\T^n\times(0,T), \\
    \tilde{c}|_{t=0} &= 0 &\quad&\text{ in }\T^n.
\end{alignat*}
Testing the latter by $\tilde{c}$ gives
\begin{align*}
    \frac{1}{2}\frac{d}{dt}\|\tilde{c}\|_{L^2(\T^n)}^2 + \Ec_\eta(\tilde{c}) + \int_{\T^n}\big(f^\prime(c_1) - f^\prime(c_2)\big)\tilde{c}\,\text{d}x = 0,
\end{align*}
where we note that $f'(c_i)\in L^2(0,T;L^2(\T^n))$ for $i=1,2$ because of \ref{ass_3} and embeddings. Moreover, the growth assumption of $f$ from \ref{ass_3} yields
\begin{align*}
    \int_{\T^n}\big(f^\prime(c_1) - f^\prime(c_2)\big)\tilde{c}\,\text{d}x \geq -\alpha\|\tilde{c}\|_{L^2(\T^n)}^2.
\end{align*}
Therefore, it holds 
\begin{align*}
    \frac{1}{2}\frac{d}{dt}\|\tilde{c}\|_{L^2(\T^n)}^2 + \Ec_\eta(\tilde{c}) \leq  \alpha\|\tilde{c}\|_{L^2(\T^n)}^2
\end{align*}
and hence the Gronwall Lemma yields $c_1=c_2$.
\end{proof}

\begin{Theorem}[\textbf{Existence, Higher Regularity and Boundedness}]\label{th_existence}
    Let the assumptions \ref{ass_1}-\ref{ass_3} hold, $T>0$ and $\eps>0$. Moreover, let $c_{0,\eps}\in H^1(\T^n)$. Then, there exists a unique solution $c_\eps$ to \eqref{eq_ACnonloc_1}--\eqref{eq_ACnonloc_2} with the regularity 
    \begin{align*}
       c_\eps \in H^1(0,T;L^2(\T^n))\cap L^\infty(0,T;H^1(\T^n))\cap C^0([0,T],L^p(\T^n))
    \end{align*}
    for all $p\in[1,\infty)$ if $n=2$ and $p\in[1,6)$ if $n=3$.
    Additionally, if $c_{0,\eps}\in H^3(\T^n)$ and $|c_{0,\eps}|\leq R_0$ with $R_0\geq 1$ as in \ref{ass_3}, then the solution $c_\eps$ has the regularity
    \begin{align*}
        c_\eps \in H^1(0,T;L^2(\T^n))\cap L^\infty(0,T;H^3(\T^n)) \cap C^1([0,T],C^0(\T^n))\cap C^0([0,T],H^2(\T^n))
    \end{align*}
    and $|c_\eps|$ is uniformly bounded by $R_0$.
\end{Theorem}
\begin{proof}
    For simplicity we set $\eps=1$ in the proof and omit the $\eps$ in the notation. All the arguments can be done analogously for the case of arbitrary $\eps>0$. Moreover, $\eta>0$ is fixed in the following.
    
    We prove this theorem using a Galerkin ansatz, in particular we project the equations and solution spaces onto finite dimensional subspaces of $H^1(\T^n)$. 
    For the approximation scheme, we consider the finite-dimensional subspaces 
    \[
    V_N := \text{span}\{1, \cos(k\cdot.), \sin(k\cdot.): k=(k_1,\ldots,k_n)\in\Z^n, |k_1|,\ldots,|k_n|\leq N\}, \quad N\in\N,
    \]
    generated by trigonometric functions. One can directly show that the Laplacian is diagonizable on $V_N$ and we denote the eigenfunctions by $w_1,...,w_{\dim V_N}$. Since $V_{N_1}\subseteq V_{N_2}$ for $N_1\leq N_2$ this is consistent for all $N\in\N$ and it is well-known that $(w_i)_{i\in\N}$ form an orthonormal basis of $L^2(\T^n)$. We denote with $(\lambda_i)_{i\in\N}$ the corresponding eigenvalues. Moreover, note that we have the following representation of the $L^2$-projection for (real-valued) $f\in L^2(\T^n)$ and all $N\in\N$:
    \begin{align}\label{eq_projection}
        P_N f:=P^{L^2}_{V_N}f=S_Nf=D_N\ast f,
    \end{align}
    where 
    \[
    S_Nf:=\sum_{k\in\Z^n,|k_1|,...,|k_N|\leq N}\frac{1}{(2\pi)^n}\hat{f}_k e^{i k\cdot .}\quad\text{ and }\quad D_N:=\sum_{k\in\Z^n,|k_1|,...,|k_N|\leq N}e^{i k\cdot.}
    \]
    for all $N\in\N$ and $f\in L^1(\T^n)$. Properties of $S_N$, especially for convergence, are well-known from Fourier-Analysis, cf. \cite{Grafakos}. 
    Recalling that convolution operators commute, we observe that $\mathcal{L}_\eta P_N u = P_N\mathcal{L}_\eta u$ for all $u\in L^2(\T^n)$. Moreover, we note that $\sup_{N\in\N}\|P_N\|_{\mathcal{L}(L^p(\T^n))}< \infty$, cf. \cite[Theorem 4.1.8]{Grafakos}.
    In particular, $\textup{span}\{w_i:i\in\N\}$ is dense in $H^k(\T^n)$ for all $k\in\N$. Finally, for vector-valued $f\in L^2(\T^n)^m, m\in\N$ equation \eqref{eq_projection} is understood component-wise.\\
    \newline
    1. \textbf{Finite dimensional approximation.} In this step, we determine approximate solutions $c^N$ of \eqref{eq_ACnonloc_1}--\eqref{eq_ACnonloc_2} of the form
    \begin{align*}
        c^N = \sum_{i=1}^N g_i^Nw_i,
    \end{align*}
    where $\mathbf{g}^N := (g_i^N)_{i=1}^N : [0,T] \rightarrow \R$ will be continuously differentiable. In particular, it holds $c^N(t) \in V_N$ for all $t\in[0,T]$. We will construct $\mathbf{g}^N$ such that $c^N$ is a solution of system 
    \begin{align}\label{eq_approx_system}
        \int_{\T^n}\partial_t c^N(t) w_i\,\text{d}x + \int_{\T^n}\Lc_\eta c^N(t) w_i\,\text{d}x + \int_{\T^n}f^\prime(c^N(t))w_i\,\text{d}x = 0
    \end{align}
    for all $t\in[0,T]$ and $i=1,\ldots,N$ together with the initial condition
    \begin{align}\label{eq_approx_initial}
        c^N(0) = P_Nc_0=\sum_{i=1}^N(c_0,w_i)_{L^2(\T^n)}w_i.
    \end{align}
    Since the functions $(w_i)_{i\in\N}$ are an orthonormal basis of $L^2(\T^n)$, the approximate system \eqref{eq_approx_system}-\eqref{eq_approx_initial} is equivalent to the following ordinary differential equation for $\mathbf{g}^N$:
    \begin{align}\label{eq_Ode1}
        \frac{d}{dt}g_i^N(t) = -\sum_{k=1}^Ng_k^N(t)\int_{\T^n}\Lc_\eta w_k w_i\,\text{d}x - \int_{\T^n}f^\prime\left(\sum_{k=1}^N g_k^N(t)w_k\right)w_i\,\text{d}x
    \end{align}
    together with the initial condition
    \begin{align}\label{eq_Ode2}
        g_i^N(0) = (c_0,w_i)_{L^2}
    \end{align}
    for all $i=1,\ldots,N$. Since the right-hand side in \eqref{eq_Ode1} depends continuously on $\mathbf{g}^N$, Peano's Theorem guarantees the existence of a local $C^1$-solution of this initial value problem on a right-maximal interval $[0,T_N^*)\cap [0,T]$ for some $T_N^*>0$. In the next two steps, we will derive an energy estimate and show global existence of the Galerkin approximation. Therefore we show uniform bounds on $[0,T_N]$ for $T_N\in [0,T_N^*)\cap [0,T]$ arbitrary.\\
\newline
2. \textbf{Energy estimate.} First of all, we define 
\begin{align*}
    E^N(t) := \frac{1}{4}\int_{\T^n}\int_{\T^n}J_\eta(x-y)\big|c^N(x,t) - c^N(y,t)\big|^2\,\text{d}y\,\text{d}x + \int_{\T^n}f(c^N(t))\,\text{d}x
\end{align*}
for all $t\in[0,T_N]$. Then, we compute
\begin{align*}
    \frac{d}{dt}E^N(t) 
    &= \frac{1}{2}\int_{\T^n}\int_{\T^n}J_\eta(x-y)\big(c^N(x,t)-c^N(y,t)\big)\big(\partial_t c^N(x,t)-\partial_t c^N(y,t)\big)\,\text{d}y\,\text{d}x \\
    &+ \int_{\T^n}f^\prime(c^N(t))\partial_t c^N(t)\,\text{d}x = -\int_{\T^n}\big|\partial_t c^N(t)\big|^2\,\text{d}x \leq 0,
\end{align*}
where we used that $g_1^N,\ldots,h_N^N$ are $C^1$-functions and for the second equality we multiplied \eqref{eq_approx_system} by $(g_i^N)'(t)$ and summed over all $i=1,\ldots,N$. Integrating the inequality above with respect to time yields for all $s\in[0,T_N]$ and all $N\in\mathbb{N}$
\begin{align}
    E^N(s) + \int_0^s\int_{\T^n}\left|\partial_t c^N(x,t)\right|^2\,\text{d}x\,\text{d}t \leq E^N(0).
\end{align}
In order to establish uniform bounds for our solutions, we need to verify that $E^N(0)$ can be bounded by some constant $C > 0$, which does not depend on $N$.

First of all, because of $c_0\in H^1(\T^n)$, the initial condition \eqref{eq_approx_initial} and the convolution representation in \eqref{eq_projection} it follows that the sequence $(c^N(0))_{N\in\mathbb{N}} \subseteq H^1(\T^n)$ is bounded. By \cite[Theorem 1]{BourgainBrezisMironescu}, there exists a constant $C>0$ such that 
\begin{align*}
     \frac{1}{4}\int_{\T^n}\int_{\T^n}J_\eta(x-y)\big|c^N(x,0) - c^N(y,0)\big|^2\,\text{d}y\,\text{d}x = \Ec_\eta(c^N(0)) \leq C\|\nabla c^N(0)\|_{L^2(\T^n)}^2.
\end{align*}
Moreover, the assumptions for $f$ from \ref{ass_3} and the Sobolev embedding theorem imply that
\begin{align*}
    \int_{\T^n}f(c^N(0))\,\text{d}x \leq C\left(1+\|c^N(0)\|_{L^4(\T^n)}^4\right)
\end{align*}
is bounded uniformly in $N\in\mathbb{N}$. Finally, we obtain the following energy estimate
\begin{align}\label{energy_estimate}
    E^N(s) + \int_0^s\int_{\T^n}\big|\partial_t c^N(x,t)\big|^2\,\text{d}x\text{d}t \leq E^N(0) \leq C
\end{align}
for all $s\in[0,T_N]$ and all $N\in\mathbb{N}$, where $C$ is independent of $T_N\in [0,T_N^*)\cap [0,T]$ and $N\in\N$.\\
\newline
3. \textbf{Uniform estimates.} From the energy inequality \eqref{energy_estimate} we obtain the uniform estimate $\|\partial_t c^N\|_{L^2(0,T_N;L^2(\T^n))} \leq C$. In the next step we prove that $(c^N)_{N\in\mathbb{N}}\subseteq L^\infty(0,T_N;H^1(\T^n))$ is bounded. To this end, we first test \eqref{eq_approx_system} by $\Delta c^N$. More precisely, we multiply \eqref{eq_approx_system} by $\lambda_i g_i^N(t)$ and sum over all $i=1,\ldots,N$. This yields
\begin{align*}
    \frac{1}{2}\frac{d}{dt}\|\nabla c^N\|_{L^2(\T^n)^n}^2 + \frac{1}{4}\int_{\T^n}\int_{\T^n}J_\eta(x-y)\big|\nabla c^N(x,t) - \nabla c^N(y,t)\big|^2\,\text{d}y\,\text{d}x \\
    + \int_{\T^n}\nabla f'(c^N)\cdot\nabla c^N\,\text{d}x = 0.
\end{align*}
Due to the assumption on $f''$ from \ref{ass_3}, it holds
\begin{align*}
    \int_{\T^n}\nabla f^\prime(c^N)\cdot\nabla c^N\,\text{d}x \geq -\alpha\|\nabla c^N\|_{L^2(\T^n)^n}^2.
\end{align*}
Thus, we can apply Gronwall's inequality to conlude that $\big(\nabla c^N\big)_{N\in\mathbb{N}} \subseteq L^\infty(0,T_N;L^2(\T^n)^n)$ is bounded. Next, we test \eqref{eq_approx_system} by $c^N$. On the Galerkin level, we multiply \eqref{eq_approx_system} by $g_i^N(t)$ and sum over all $i=1,\ldots,N$. This yields
\begin{align*}
    \frac{1}{2}\frac{d}{dt}\|c^N\|_{L^2(\T^n)}^2 + \frac{1}{4}\int_{\T^n}\int_{\T^n}J_\eta(x-y)\big|c^N(x,t) - c^N(y,t)\big|^2\,\text{d}y\,\text{d}x + \int_{\T^n}f^\prime(c^N)c^N\,\text{d}x=0.
\end{align*}
In order to control the third term on the right-hand side, we use the growth assumption of $f$ from \ref{ass_3}, the Hölder inequality as well as the Gagliardo-Nirenberg inequality. Hence we obtain
\begin{align*}
    \int_{\T^n}f^\prime(c^N)c^N\,\text{d}x &\leq \|f^\prime(c^N)\|_{L^2(\T^n)}\|c^N\|_{L^2(\T^n)} \\
    &\leq C\|c^N\|_{L^2(\T^n)} + C\|c^N\|_{L^6(\T^n)}^2\|c^N\|_{L^2(\T^n)} \\ 
    &\leq C\|c^N\|_{L^2(\T^n)} + C\|\nabla c^N\|_{L^2(\T^n)^n}^{2\frac{n}{3}}\|c^N\|_{L^2(\T^n)}^{3-2\frac{n}{3}}.
\end{align*}
Since $\big(\nabla c^N\big)_{N\in\mathbb{N}}$ is bounded in $L^\infty(0,T_N;L^2(\T^n)^n)$, we apply Young's inequality and obtain from Gronwall that $(c^N)_{N\in\mathbb{N}}\subseteq L^\infty(0,T_N;L^2(\T^n))$ is bounded. 

Altogether, the sequence $(c^N)_{N\in\mathbb{N}}$ is bounded in $C^0([0,T_N],H^1(\T^n))\cap H^1(0,T_N;L^2(\T^n))$ uniformly for all $T_N\in [0,T_N^*)\cap [0,T]$ and $N\in\N$. In particular, this yields that the solution $\mathbf{g}^N$ is bounded on $[0,T_N]$ uniformly for all $T_N\in[0,T_N^\ast)\cap[0,T]$ and all $N\in\N$. An extension argument yields that $\mathbf{g}^N$ and hence $c^N$ exist globally on $[0,T]$ and $T_N=T$ can be chosen above.\\
\newline
4. \textbf{Passage to the limit.} Due to the uniform estimates derived in the step before, there exists a subsequence, again denoted by $(c^N)_{N\in\mathbb{N}}$, such that for $N\rightarrow\infty$
\begin{align*}
    c^N &\rightharpoonup c &&\text{ in }H^1(0,T;L^2(\T^n)),\\
    c^N &\rightarrow c &&\text{ in }C^0([0,T];L^p(\T^n)),
\end{align*}
for some $c\in H^1(0,T;L^2(\T^n))\cap C^0([0,T];L^p(\T^n))$ and for all $p\in[1,\infty)$ if $n=2$ and all $p\in[1,6)$ if $n=3$,
where the last convergence follows by the Aubin-Lions-Simon Lemma. We need to show equation \eqref{eq_ACnonloc_1} for $c$.

Let $\xi\in C^\infty_0(0,T)$. We multiply \eqref{eq_approx_system} by $\xi(t)$ and integrate with respect to time. Hence
\begin{align*}
    \int_0^T\int_{\T^n}\partial_tc^Nw_i\xi\,\text{d}x\,\text{d}t + \int_0^T\int_{\T^n}\Lc_\eta c^Nw_i\xi\,\text{d}x\,\text{d}t + \int_0^T\int_{\T^n}f^\prime(c^N)w_i\xi \,\text{d}x\,\text{d}t = 0.
\end{align*}
For the first term on the right-hand side, we use the weak convergence $c^N \rightharpoonup c$ in $H^1(0,T;L^2(\T^n))$ for $N\rightarrow\infty$ to conclude that 
\begin{align*}
     \int_0^T\int_{\T^n}\partial_tc^Nw_i\xi\,\text{d}x\,\text{d}t \rightarrow  \int_0^T\int_{\T^n}\partial_tc\,w_i\xi\,\text{d}x\,\text{d}t.
\end{align*}
For the second term we prove in the following that $\Lc_\eta c^N \rightarrow \Lc_\eta c$ in $L^2(0,T;L^2(\T^n))$ for $N\rightarrow\infty$. By definition and Young's convolution inequality, we have
\begin{align*}
    \|&\Lc_\eta c^N - \Lc_\eta c\|_{L^2(0,T;L^2(\T^n))}^2 \\
    &\leq \int_0^T\|J_\eta * (c^N(t)-c(t))\|_{L^2(\T^n)}^2\,\text{d}t+ \int_0^T\|(J_\eta * 1)(c^N(t)-c(t))\|_{L^2(\T^n)}^2\,\text{d}t \\
    &\leq \int_0^T\|J_\eta\|_{L^1(\T^n)}^2\|(c^N(t)-c(t))\|_{L^2(\T^n)}^2\,\text{d}t + \int_0^T\|(J_\eta * 1)(c^N(t)-c(t))\|_{L^2(\T^n)}^2\,\text{d}t.
\end{align*}
Since the term $J_\eta*1$ is constant on the torus, the assertion follows from the strong convergence $c^N\rightarrow c$ in $C^0([0,T];L^2(\T^n))$ for $N\rightarrow\infty$. Altogether, this implies for $N\rightarrow\infty$ that
\begin{align*}
    \int_0^T\int_{\T^n}\Lc_\eta c^Nw_i\xi\,\text{d}x\,\text{d}t \rightarrow \int_0^T\int_{\T^n}\Lc_\eta c\,w_i\xi\,\text{d}x\,\text{d}t.
\end{align*}
Finally, for the last term we apply the General Lebesgue Dominated Convergence Theorem. Due to the growth assumption of $f'$ from \ref{ass_3} it holds
\begin{align*}
    |f^\prime(c^N)w_i| \leq C|w_i|\left(1+|c^N|^3\right) =: h_N.
\end{align*}
We will show $h_N\rightarrow h:= C|w_i|\big(1+|c|^3\big)$ in $L^1(\T^n\times(0,T))$. We have 
\begin{align*}
    \|h_N-h\|_{L^1(\T^n\times(0,T))} = C\int_0^T\int_{\T^n}|w_i|\Big||c^N|^3-|c|^3\Big|\,\text{d}x\,\text{d}t.
\end{align*}
Defining the function $b:\R^+\rightarrow\R^+:x\mapsto x^3$, the mean value theorem guarantees the existence of some $\mu(x,t)\in(0,1)$ such that for $\zeta(x,t) := \mu(x,t)|c(x,t)| + (1-\mu(x,t))|c^N(x,t)|$ it holds
\begin{align*}
    \left||c^N(x,t)|^3-|c(x,t)|^3\right| = |b^\prime(\zeta(x,t))|\;\left||c^N(x,t)|-|c(x,t)|\right|
\end{align*}
for a.e.~$(x,t)\in\T^n\times(0,T)$. This yields
\begin{align*}
    \|h_N-h\|_{L^1(\T^n\times(0,T))} &= C\int_0^T\int_{\T^n}|w_i|\Big||c^N|^3-|c|^3\Big|\,\text{d}x\,\text{d}t\\
    &\leq \tilde{C}\int_0^T\int_{\T^n}|w_i|\Big(|c^N|^2 + |c|^2\Big)\big||c^N|-|c|\big|\,\text{d}x\,\text{d}t\\
    &\leq \tilde{C}\|w_i\|_{L^6(\T^n)}\int_0^T\Big(\|c^N\|_{L^4(\T^n)}^2 + \|c\|_{L^4(\T^n)}^2\Big)\|c^N - c\|_{L^3(\T^n)}\,\text{d}t.
\end{align*}
Using that $(c^N(t))_{N\in\mathbb{N}}\subseteq L^4(\T^n)$ and $c(t)\in L^4(\T^n)$ are uniformly bounded for all $t\in[0,T]$, the strong convergence $c^N\rightarrow c$ in $C^0([0,T];L^3(\T^n))$ implies that the right-hand side vanishes as $N\rightarrow\infty$. This shows $h_N\rightarrow h$ in $L^1(\T^n\times(0,T))$ for $N\rightarrow\infty$. In particular, this implies 
\begin{align*}
    \int_0^T\int_{\T^n}f^\prime(c^N)w_i\xi \,\text{d}x\,\text{d}t \rightarrow \int_0^T\int_{\T^n}f^\prime(c)w_i\xi \,\text{d}x\,\text{d}t
\end{align*}
for $N\rightarrow\infty$. Finally, we obtain 
\begin{align*}
    \int_0^T\int_{\T^n}\partial_tcw_i\xi\,\text{d}x\,\text{d}t + \int_0^T\int_{\T^n}\Lc_\eta cw_i\xi\,\text{d}x\,\text{d}t + \int_0^T\int_{\T^n}f^\prime(c)w_i\xi \,\text{d}x\,\text{d}t = 0
\end{align*}
for all $i\in\mathbb{N}$ and all $\xi\in C^\infty_0(0,T)$. Since $\text{span}\{w_i
:\,i \in \mathbb{N}\}$ is dense in $H^1(\T^n)$ we obtain
\begin{align*}
    \int_0^T\int_{\T^n}\partial_tcw\xi\,\text{d}x\,\text{d}t + \int_0^T\int_{\T^n}\Lc_\eta cw\xi\,\text{d}x\,\text{d}t + \int_0^T\int_{\T^n}f^\prime(c)w\xi \,\text{d}x\,\text{d}t = 0
\end{align*}
for all $w\in H^1(\T^n)$ and all $\xi\in C^\infty_0(0,T)$. Therefore, the Fundamental Lemma of Calculus of Variations implies that \eqref{eq_ACnonloc_1} holds for $c$.

Finally, we show that the initial condition holds. We have already seen that $c^N(0) \rightarrow c_0$ in $H^1(\T^n) \hookrightarrow L^2(\T^n)$ for $N\rightarrow\infty$. Since we also have the convergence $c^N\rightarrow c$ in $C^0([0,T];L^2(\T^n))$, we conclude that $c(0) = c_0$ in $L^2(\T^n)$ for $N\rightarrow\infty$.\\
\newline
5. \textbf{Higher order estimates.} In this part we prove higher regularity for the solution obtained in the steps before provided that the initial value satisfies $c_0\in H^3(\T^n)$ and $|c_0|\leq R_0$ with $R_0\geq 1$ as in \ref{ass_3}. In the end we prove that the solution is confined to $[-R_0,R_0]$ in this situation and hence we can change the potential outside of this interval suitably. Therefore from now on we assume without loss of generality that the potential $f$ and its derivatives are uniformly bounded.  

We first prove that $c\in L^\infty(0,T;H^2(\T^n))$. For the proof, we again use the Galerkin scheme introduced in step 1. We test equation \eqref{eq_approx_system} by $\Delta^2 c^N$, which means, on the Galerkin level, that we multiply \eqref{eq_approx_system} by $\lambda_i^2g_i^N(t)$ and sum over all $i=1,\ldots,N$. This yields 
\begin{align*}
    \frac{1}{2}\frac{d}{dt}\|\Delta c^N\|_{L^2(\T^n)}^2 + \frac{1}{4}\int_{\T^n}\int_{\T^n}J_\eta(x-y)\big|\Delta c^N(x,t) - \Delta c^N(y,t)\big|^2\,\text{d}y\,\text{d}x \\
    + \int_{\T^n}\Delta f^\prime(c^N)\Delta c^N\,\text{d}x=0.
\end{align*}
For the third term on the right-hand side, we observe
\begin{align*}
    \int_{\T^n}\Delta f^\prime(c^N)\Delta c^N\,\text{d}x &=  \int_{\T^n} f^{\prime\prime\prime}(c^N)|\nabla c^N|^2\Delta c^N\,\text{d}x + \int_{\T^n}f^{\prime\prime}(c^N)|\Delta c^N|^2\,\text{d}x \\
    &\geq \int_{\T^n} f^{\prime\prime\prime}(c^N)|\nabla c^N|^2\Delta c^N\,\text{d}x - \alpha\|\Delta c^N\|_{L^2(\T^n)}^2.
\end{align*}
By our assumption on the potential $f$, it holds $\|f^{(k)}\|_{L^\infty(\R)}\leq C$ for all $k\geq 3$. Thus, the remaining term can be controlled as follows:
\begin{align*}
    \left|\int_{\T^n} f'''(c^N)|\nabla c^N|^2\Delta c^N\,\text{d}x \right|\leq \|f'''\|_\infty\|\nabla c^N\|_{L^4(\T^n)}^2\|\Delta c^N\|_{L^2(\T^n)}.
\end{align*}
In order to control $\|\nabla c^N\|_{L^4(\T^n)}$, note that by \eqref{eq_approx_system} 
\[
0=P_N(\partial_tc^N+\Lc_\eta c^N+f'(c^N))=\partial_tc^N+\Lc_\eta c^N+P_N(f'(c^N)),
\]
where we used that $P_N$ and $\Lc_\eta$ commutes due to the representation \eqref{eq_projection} and properties of convolutions.
We differentiate the latter equation with respect to the space variable and test by $P_N(|\nabla c^N|^{p-2}\nabla c^N)$, $p>2$. This yields
\begin{align}
    \int_{\T^n}\partial_t \nabla c^N \cdot P_N(|\nabla c^N|^{p-2}\nabla c^N)\,\text{d}x + \int_{\T^n}\mathcal{L}_\eta \nabla c^N \cdot|\nabla c^N|^{p-2}\nabla c^N\,\text{d}x \nonumber\\
    +\int_{\T^n}f''(c^N)\nabla c^N \cdot P_N(|\nabla c^N|^{p-2}\nabla c^N)\,\text{d}x = 0, \label{calc_higher_reg_gradient}
\end{align}
where we used $\Lc_\eta \nabla c^N=\Lc_\eta P_N\nabla c^N=P_N\Lc_\eta\nabla c^N$ for the second term. This follows from the fact that convolution operators commute. For the first term, we use $\partial_t \nabla c^N \in (V_N)^n$ to obtain
\begin{align*}
    \int_{\T^n}\partial_t \nabla c^N \cdot P_N(|\nabla c^N|^{p-2}\nabla c^N)\,\text{d}x = \frac{d}{dt}\int_{\T^n}\frac{1}{p}|\nabla c^N|^p\,\text{d}x.
\end{align*}
For the term involving the nonlocal operator, we use a symmetry argument to conclude
\begin{align*}
    &\int_{\T^n}\mathcal{L}_\eta \nabla c^N \cdot |\nabla c^N|^{p-2}\nabla c^N\,\text{d}x \\
    &= \int_{\T^n}\int_{\T^n}J_\eta(x-y)\big(\nabla c^N(x)-\nabla c^N(y)\big)\cdot \nabla b(\nabla c^N(x))\,\text{d}y\,\text{d}x \\
    &=\frac{1}{2}\int_{\T^n}\int_{\T^n}J_\eta(x-y)\big(\nabla c^N(x)-\nabla c^N(y)\big)\cdot \left(\nabla b(\nabla c^N(x)) - \nabla b(\nabla c^N(y))\right)\,\text{d}y\,\text{d}x,
\end{align*}
where the function $b$ is defined as $b(z) := \frac{|z|^p}{p}$ for all $z\in\R^n$. Since $\nabla b$ is monotone and since $J_\eta$ is nonnegative, this implies that 
\begin{align*}
    \int_{\T^n}\mathcal{L}_\eta \nabla c^N \cdot |\nabla c^N|^{p-2}\nabla c^N\,\text{d}x \geq 0.
\end{align*}
Finally, for the last term in \eqref{calc_higher_reg_gradient}, we use Hölder's inequality to get
\begin{align*}
    \left|\int_{\T^n}f^{\prime\prime}(c^N)\nabla c^N \cdot P_N(|\nabla c^N|^{p-2}\nabla c^N)\,\text{d}x\right|
    \leq C\|\nabla c^N\|_{L^p(\T^n)}\|P_N(|\nabla c^N|^{p-2}\nabla c^N)\|_{L^{p^\prime}(\T^n)}.
\end{align*}
From Fourier Analysis and \eqref{eq_projection} it is well-known that $\sup_{N\in\N}\|P_N\|_{\mathcal{L}(L^{p^\prime}(\T^n))}\leq K$ for every $1<p<\infty$, where $K>0$ is independent of $N\in\N$. Moreover, we observe that $\||\nabla c^N|^{p-2}\nabla c^N\|_{L^{p^\prime}(\T^n)} = \|\nabla c^N\|_{L^p(\T^n)}^{p-1}$.
Hence \eqref{calc_higher_reg_gradient} yields
\begin{align*}
    \frac{d}{dt}\int_{\T^n}\frac{1}{p}|\nabla c^N|^p\,\text{d}x 
    +  \int_{\T^n}\Lc_\eta\nabla c^N\cdot\nabla b(\nabla c^N)\,\text{d}x \leq C\|\nabla c^N\|_{L^p(\T^n)}^p.
\end{align*}
Using the relation $c^N(0) = P_N c_0$ and \eqref{eq_projection}, we conclude that the sequence $(c^N(0))_{N\in\N}$ is bounded in $H^3(\T^n)$ provided that $c_0\in H^3(\T^n)$.

Thus by Gronwall's inequality the sequence $(\nabla c^N)_{N\in\mathbb{N}} \subseteq L^\infty(0,T;L^p(\T^n)^n)$ is bounded for all $p>2$. Therefore it holds 
\begin{align*}
    \int_{\T^n} f^{\prime\prime\prime}(c^N)|\nabla c^N|^2\Delta c^N\,\text{d}x \leq C(1+\|\Delta c^N\|_{L^2(\T^n)}^2)
\end{align*}
and thus, again by Gronwall's inequality, $(c^N)_{N\in\mathbb{N}}\subseteq L^\infty(0,T;H^2(\T^n))$ is bounded since the second order derivatives can be controlled with the Laplacian via integration by parts.

In the next step, we prove $c\in L^\infty(0,T;H^3(\T^n))$. To this end, we test equation \eqref{eq_approx_system} by $\Delta^3c^N$. More precisely, we multiply \eqref{eq_approx_system} by $\lambda_ig_i^N(t)$ and sum over all $i=1,\ldots,N$. This yields
\begin{align*}
    \frac{1}{2}\frac{d}{dt}\|\nabla\Delta c^N\|_{L^2(\T^n)^n}^2 + \frac{1}{4}\int_{\T^n}\int_{\T^n}J_\eta(x-y)\big|\nabla\Delta c^N(t,x) - \nabla\Delta c^N(t,y)\big|^2\,\text{d}y\text{d}x \\
    + \int_{\T^n}\nabla\Delta f^\prime(c^N)\cdot\nabla\Delta c^N\,\text{d}x=0.
\end{align*}
We need to control the last term on the right-hand side. First of all, the chain rule gives
\begin{equation}
\begin{aligned}\label{f^3_chain_rule}
    \nabla\Delta f^\prime(c^N) = &f^{(4)}(c^N)|\nabla c^N|^2\nabla c^N + 2f^{\prime\prime\prime}(c^N)D^2c^N\nabla c^N \\
    &+ f^{\prime\prime\prime}(c^N)\Delta c^N\nabla c^N + f^{\prime\prime}(c^N)\nabla\Delta c^N.
\end{aligned}
\end{equation}
In order to estimate the terms we use that we assumed without loss of generality that the derivatives of $f$ are uniformly bounded. Moreover,
\begin{align*}
    \left\| |\nabla c^N|^2\nabla c^N\right\|_{L^2(\T^n)^n} 
    &=\|\nabla c^N\|_{L^6(\T^n)^n}^3 \leq C, \\
    \left\|D^2c^N\nabla c^N\right\|_{L^2(\T^n)^n} 
    &\leq \|D^2 c^N\|_{L^2(\T^n)^{n\times n}}\|\nabla c^N\|_{L^\infty(\T^n)^n}  
    \leq C\|c^N\|_{H^2(\T^n)}\|c^N\|_{H^3(\T^n)} \\
    &\leq C\|c^N\|_{H^2(\T^n)}\left(1+\|\nabla\Delta c^N\|_{L^2(\T^n)^n}\right), \\
    \left \|\Delta c^N\nabla c^N\right\|_{L^2(\T^n)^n} &\leq 
    C\|c^N\|_{H^2(\T^n)}\left(1+\|\nabla\Delta c^N\|_{L^2(\T^n)^n}\right).
\end{align*}
Therefore we can control the corresponding parts in the integral above using Young's inequality.
For the last term, we use the growth condition of $f$ from \ref{ass_3}, which implies
\begin{align*}
    \int_{\T^n}f^{\prime\prime}(c^N)|\nabla\Delta c^N|^2\,\text{d}x \geq -\alpha\|\nabla\Delta c^N\|_{L^2(\T^n)^n}^2.
\end{align*}
Altogether, Gronwall shows that $(c^N)_{N\in\mathbb{N}}\subseteq L^\infty(0,T;H^3(\T^n))$ is bounded. 

From the previous estimates we obtain $c\in H^1(0,T;L^2(\T^n))\cap L^\infty(0,T;H^3(\T^n))$. Now, due to the Aubin-Lions-Simon Lemma, we get
\begin{align*}
    H^1(0,T;L^2(\T^n))\cap L^\infty(0,T;H^3(\T^n)) \hookrightarrow C^0([0,T];H^2(\T^n)),
\end{align*}
which in particular yields $c\in C^0([0,T],C^0(\T^n))$. Finally, from the equation \eqref{eq_ACnonloc_1} for $c$
we obtain $c\in C^1([0,T];C^0(\T^n))$.\\
\newline
6. \textbf{Boundedness.} Finally, we prove that the solution $c\in C^1([0,T];C^0(\T^n))$ of \eqref{eq_ACnonloc_1}--\eqref{eq_ACnonloc_2} is bounded by $R_0$ provided that $|c_0|\leq R_0$, where $R_0\geq 1$ is from \ref{ass_3}. Assume that there exists $(x_0,t_0)\in \T^n\times[0,T]$ such that $c$ attains its maximum in $(x_0,t_0)$ with $c(x_0,t_0) > R_0$. Then 
$\partial_tc(x_0,t_0)\geq 0$ and it holds
    \begin{align*}
        \partial_tc(x_0,t_0) + \Lc_\eta c(x_0,t_0) + f^\prime(c(x_0,t_0)) >0,
    \end{align*}
    since $f^\prime(c(x_0,t_0))>0$ if $c(x_0,t_0) > R_0$ and 
    \begin{align*}
        \mathcal{L}_\eta c(x_0,t_0) = \int_{\T^n}J_\eta(x_0-y)(c(x_0,t_0)-c(y,t_0))\,dy \geq 0,
    \end{align*}
    since $J_\eta\geq 0$.
    This is a contradiction since $c$ solves \eqref{eq_ACnonloc_1}. 
    Analogously, one shows $c \geq -R_0$. 
    \newline
    
    This concludes the proof of Theorem \ref{th_existence}.
\end{proof}

\subsection{Uniform $H^3$-Bounds}\label{sec_nonlocAC_bound}
\begin{Theorem}[\textbf{Uniform $H^3$-Bounds}]\label{th_H^3_estimate}
    Let $T>0$ be fixed. For any $\eps\in(0,1)$, let $c_\eps$ be the solution to \eqref{eq_ACnonloc_1}-\eqref{eq_ACnonloc_2} from Theorem \ref{th_existence} for initial values $c_{0,\varepsilon}\in H^3(\T^n)$ such that $|c_{0,\eps}|\leq R_0$ with $R_0\geq 1$ as in \ref{ass_3} and such that with some constant $K>0$ independent of $\eps\in(0,1)$
    \[
    \|c_{0,\eps}\|_{H^1(\T^n)}\leq\frac{K}{\eps^2},\quad \|c_{0,\eps}\|_{H^2(\T^n)}\leq \frac{K}{\eps^6}\quad\text{ and }\quad \|c_{0,\eps}\|_{H^3(\T^n)}\leq\frac{K}{\eps^{16}}.
    \]
    Then there exist constants $\eps_0\in(0,1)$, $c,C>0$ independent of $c_\eps,\eps,\eta$ such that for all $0<\eta\leq c\eps^4$, 
    $\eps\in(0,\eps_0)$ it holds
    \[
    \|c_\eps\|_{L^\infty(0,T;H^3(\T^n))} \leq \frac{C}{\eps^{16}}.
    \]
\end{Theorem}
\begin{proof}
Let $c_\eps$ be the solution of \eqref{eq_ACnonloc_1}-\eqref{eq_ACnonloc_2} according to Theorem \ref{th_existence}. We have to control the $H^3$-norm of $c_\eps$ in dependence of $\eps$. To this end, we proceed in several steps.
\newline
\newline
1. \textbf{$H^1$-estimate.} We test \eqref{eq_ACnonloc_1} with $-\Delta c_\eps$. After integration by parts, we obtain for all $t\in[0,T]$
\begin{align*}
    \frac{1}{2}\|\nabla c_\eps(t)\|_{L^2(\T^n)^n}^2 + \int_0^t\Ec_\eta(\nabla c_\eps)\,ds \leq \frac{1}{2}\|\nabla c_{0,\eps}\|_{L^2(\T^n)^n}^2+\frac{\alpha}{\eps^2}\int_0^t\|\nabla c_\eps\|_{L^2(\T^n)^n}^2\,ds
\end{align*}
where we also used the condition for $f''$ from \ref{ass_3}. Next, we apply Theorem \ref{th_nonloc_ehrling} (Nonlocal Ehrling lemma) which yields for all $R\geq 1$, $0<\eta\leq\frac{1}{R}$ that
\begin{align*}
    &\frac{1}{2}\|\nabla c_\eps(t)\|_{L^2(\T^n)^n}^2 + \int_0^t\Ec_\eta(\nabla c_\eps)\,ds\\ 
    &\leq 
    \frac{1}{2}\|\nabla c_{0,\eps}\|_{L^2(\T^n)^n}^2
    +\int_0^t\frac{C\alpha}{\eps^2 R^2}\Ec_\eta(\nabla c_\eps) + \frac{C\alpha}{\eps^2}R^2\|c_\eps\|_{L^2(\T^n)}^2\,ds. 
\end{align*}
We choose $R=2\frac{\sqrt{C\alpha}}{\eps}$ and $\eps$ small such that $R\geq 1$. Thus the Gronwall Lemma yields for all $0<\eta\leq\frac{1}{2\sqrt{C\alpha}}\eps$ and $\eps$ small that
$\|\nabla c_\eps\|_{L^\infty(0,T;L^2(\T^n)^n)} \leq \frac{C}{\eps^2}$.\\
\newline
2. \textbf{$H^2$-estimate.} We test \eqref{eq_ACnonloc_1} with $\Delta^2 c_\eps$. This yields for all $t\in[0,T]$
\begin{align*}
    &\frac{1}{2}\|\Delta c_\eps(t)\|_{L^2(\T^n)}^2 + \int_0^t\Ec_\eta(\Delta c_\eps)\,ds \\
    &\leq \frac{1}{2}\|\Delta c_{0,\eps}\|_{L^2(\T^n)}^2
    +\frac{\alpha}{\eps^2}\int_0^t\|\Delta c_\eps\|_{L^2(\T^n)}^2\,ds - \frac{1}{\eps^2}\int_0^t\int_{\T^n}f^{\prime\prime\prime}(c_\eps)|\nabla c_\eps|^2\Delta c_\eps\,dx\,ds.
\end{align*}
Because of the maximum principle in Theorem \ref{th_existence}, we have that $|c_\eps|$ is uniformly bounded. Moreover, the Gagliardo-Nirenberg interpolation inequality gives for all $t\in[0,T]$
\begin{align*}
    \|\nabla c_\eps\|_{L^4(\T^n)^n} \leq C\|c_\eps\|_{L^\infty(\T^n)}^{1/2}\|c_\eps\|_{H^2(\T^n)}^{1/2} \leq C\|c_\eps\|_{H^2(\T^n)}^{1/2}
\end{align*}
for both $n\in\{2,3\}$, where we omit the time-dependence for simplicity. Then by Young's inequality
\begin{align*}
    \frac{1}{\eps^2}\int_{\T^n}f^{\prime\prime\prime}(c_\eps)|\nabla c_\eps|^2\Delta c_\eps\,\text{d}x 
    &\leq \frac{C}{\eps^2}\|\nabla c_\eps\|_{L^4(\T^n)^n}^4 + \frac{C}{\eps^2}\|\Delta c_\eps\|_{L^2(\T^n)}^2 \\ 
    &\leq \frac{C}{\eps^2}\|c_\eps\|_{H^2(\T^n)}^2 + \frac{C}{\eps^2}\|\Delta c_\eps\|_{L^2(\T^n)}^2 \\
    &\leq   \frac{C}{\eps^4} + \frac{C}{\eps^2}\|\Delta c_\eps\|_{L^2(\T^n)}^2 
\end{align*}
due to the estimate we derived in step 1. Together with Theorem \ref{th_nonloc_ehrling} (Nonlocal Ehrling Inequality), we get for all $R\geq1,0<\eta\leq\frac{1}{R}$ and all $t\in[0,T]$ 
\begin{align*}
    \frac{1}{2}\|\Delta c_\eps(t)\|_{L^2(\T^n)}^2 + \int_0^t\Ec_\eta(\Delta c_\eps)\,ds
    &\leq \frac{C}{\eps^4} + \frac{C}{\eps^4}\int_0^t\|\Delta c_\eps\|_{L^2(\T^n)}^2\,ds\\
    &\leq \frac{C}{\eps^4}
    +\int_0^t\frac{C}{\eps^4R^2}\Ec_\eta(\Delta c_\eps) + \frac{C}{\eps^4}R^2\|\nabla c_\eps\|_{L^2(\T^n)^n}^2\,ds.
\end{align*}
Choosing $R$ proportional to $\frac{1}{\eps^2}$ to absorb the $\Ec_\eta$-term, Gronwall's inequality then implies
\[
    \|\Delta c_\eps\|_{L^\infty(0,T;L^2(\T^n))} \leq \frac{C}{\eps^6}
\]
for all $0<\eta\leq c\eps^2$ and $\eps>0$ small, where $c>0$ is some fixed constant. This also controls the $L^2(\T^n)$-norm of the second derivatives due to integration by parts.\\
\newline
3. \textbf{$H^3$-estimate. } In the final step, we test \eqref{eq_ACnonloc_1} with $-\Delta^3 c_\eps$. Then we have for a.e.~$t\in[0,T]$
\begin{align*}
    &\frac{1}{2}\|\nabla\Delta c_\eps(t)\|_{L^2(\T^n)^n}^2
    -\frac{1}{2}\|\nabla\Delta c_{0,\eps}\|_{L^2(\T^n)^n}^2\\
    &+\int_0^t\Ec_\eta(\nabla\Delta c_\eps)\,ds + \frac{1}{\eps^2}\int_0^t\int_{\T^n}\nabla\Delta f'(c_\eps)\cdot\nabla\Delta c_\eps\,dx\,ds=0.
\end{align*}
For simplicity, we often omit the time-dependence in the following. The estimates are uniform with respect to time. Using the chain rule as in \eqref{f^3_chain_rule} and that $c_\eps$ is uniformly bounded, we obtain
\begin{align*}
    \big\|f^{(4)}(c_\eps)|\nabla c_\eps|^2\nabla c_\eps\big\|_{L^2(\T^n)^n} 
    \leq C\||\nabla c_\eps|^2\nabla c_\eps\|_{L^2(\T^n)^n} \leq C\|\nabla c_\eps\|_{L^6(\T^n)^n}^3. 
\end{align*}
In particular, this gives
\begin{align*}
    \left|\frac{1}{\eps^2}\int_{\T^n}f^{(4)}(c_\eps)|\nabla c_\eps|^2\nabla c_\eps\cdot\nabla\Delta c_\eps\,\text{d}x\right| &\leq \frac{C}{\eps^2}\|c_\eps\|_{H^2(\T^n)}^3\|\nabla\Delta c_\eps\|_{L^2(\T^n)} \\ 
    &\leq \frac{C}{\eps^{32}} + \frac{C}{\eps^8}\|\nabla\Delta c_\eps\|_{L^2(\T^n)}^2.
\end{align*}
Here, we used the estimates from before as well as the Poincaré inequality since
\begin{align*}
    \overline{\Delta c_\eps} = \int_{\T^n}\Delta c_\eps\,\text{d}x = 0 
\end{align*}
due to the Gauß Divergence Theorem.
For the remaining terms, we have
\begin{align*}
    \big\|f^{\prime\prime\prime}(c_\eps)D^2c_\eps\nabla c_\eps\big\|_{L^2(\T^n)^n} &\leq C\|\Delta c_\eps\|_{L^2(\T^n)}\|\nabla c_\eps\|_{H^2(\T^n)^n} \\
    &\leq \frac{C}{\eps^6}\big(\|\nabla c_\eps\|_{L^2(\T^n)^n} + \|\nabla\Delta c_\eps\|_{L^2(\T^n)^n}\big), \\
    \big\|f^{\prime\prime\prime}(c_\eps)\Delta c_\eps\nabla c_\eps\big\|_{L^2(\T^n)^n} 
    &\leq \frac{C}{\eps^6}\big(\|\nabla c_\eps\|_{L^2(\T^n)^n} + \|\nabla\Delta c_\eps\|_{L^2(\T^n)^n}\big)
\end{align*}
and therefore
\begin{align*}
    &\left|\frac{1}{\eps^2}\int_{\T^n}\Big(f^{\prime\prime\prime}(c_\eps)D^2c_\eps\nabla c_\eps + f^{\prime\prime\prime}(c_\eps)\Delta c_\eps\nabla c_\eps\Big)\cdot\nabla\Delta c_\eps\,\text{d}x \right|\\
    &\leq \frac{C}{\eps^8}\big(\|\nabla c_\eps\|_{L^2(\T^n)^n} + \|\nabla\Delta c_\eps\|_{L^2(\T^n)^n}\big)\|\nabla\Delta c_\eps\|_{L^2(\T^n)^n} \\
    &\leq \frac{C}{\eps^8}\|\nabla\Delta c_\eps\|_{L^2(\T^n)^n}^2 + \frac{C}{\eps^{12}},
\end{align*}
where we applied Young's inequality in the last step. For the last term, \ref{ass_3} implies
\begin{align*}
    \frac{1}{\eps^2}\int_{\T^n}f^{\prime\prime}(c_\eps)|\nabla\Delta c_\eps|^2\,\text{d}x \geq -\frac{\alpha}{\eps^2}\|\nabla\Delta c_\eps\|_{L^2(\T^n)^n}^2.
\end{align*}
Combining these estimates with Theorem \ref{th_nonloc_ehrling} (Nonlocal Ehrling Inequality), we obtain for all $R\geq 1$, $0<\eta\leq \min\{c\eps^2,\frac{1}{R}\}$ and for a.e.~$t\in[0,T]$
\begin{align*}
    \frac{1}{2}\|\nabla\Delta c_\eps(t)\|_{L^2(\T^n)^n}^2 
    +\int_0^t\Ec_\eta(\nabla\Delta c_\eps)\,ds
    &\leq \frac{C}{\eps^{32}}+\frac{C}{\eps^8}\int_0^t\|\nabla\Delta c_\eps\|_{L^2(\T^n)^n}^2\,ds\\
    &\leq \frac{C}{\eps^{32}}+\int_0^t\frac{C}{\eps^8R^2}\Ec_\eta(\nabla\Delta c_\eps) + \frac{C}{\eps^8}R^2\|\Delta c_\eps\|_{L^2(\T^n)}^2\,ds.
\end{align*}
Finally, we choose $R$ proportional to $\frac{1}{\eps^4}$ to absorb the $\Ec_\eta$-term. Then Gronwall's lemma, cf. \cite[Lemma II.4.10]{BoyerFabrie}, yields for all $0<\eta\leq c\eps^4$ and $\eps>0$ small, where the constant $c$ is possibly smaller than before, that
\begin{align*}
    \|\nabla\Delta c_\eps\|_{L^\infty(0,T;L^2(\T^n)^n)} \leq \frac{C}{\eps^{16}}.
\end{align*}
The $L^2$-norm of the third derivatives is then also controlled via integration by parts.
\end{proof}

\section{Proof of Theorem \ref{th_conv} (Convergence)}\label{sec_conv}
We use the notation from Theorem \ref{th_conv}. Moreover, let $c^A_\eps$ for $\eps>0$ small be the approximate solution from the local case, constructed for the evolving hypersurface $(\Gamma_t)_{t\in[0,T_0]}$ and the parameter $M\in\N$, cf.~Theorem \ref{th_local_approx}. Note that $\Gamma$ has to evolve according to mean curvature flow in order to have the remainder estimates for $c^A_\eps$ in Theorem \ref{th_local_approx} for the local Allen-Cahn equation available. For the latter we set $r^A_\eps:=\partial_tc^A_\eps-\Delta c^A_\eps+\frac{1}{\eps^2}f'(c^A_\eps)$. Finally, let $\beta\geq 0$ be fixed (to be chosen later) and set $g_\beta(t):=e^{-\beta t}$ for all $t\geq 0$. We investigate the validity of the estimates
\begin{align}\begin{split}
    \sup_{t\in[0,T]}\|g_\beta\overline{c}_\eps(t)\|_{L^2(\T^n)}^2 + \|g_\beta\nabla\overline{c}_\eps\|_{L^2(\T^n\times(0,T)\setminus\Gamma(\delta))}^2 &\leq 2R^2\eps^{2M+1},\label{eq_conv_proof_estimate}\\
    \|g_\beta\nabla_\tau\overline{c}_\eps\|_{L^2(\T^n\times(0,T)\cap\Gamma(\delta))}^2 + \eps^2 \|g_\beta\partial_n\overline{c}_\eps\|_{L^2(\T^n\times(0,T)\cap\Gamma(\delta))}^2 &\leq 2R^2\eps^{2M+1},\end{split}
\end{align}
where $\eps,R,M>0$ and $T\in(0,T_0]$. Now let us define
\[
T_{\eps,\beta,R}:=\sup\{T\in(0,T_0]:\eqref{eq_conv_proof_estimate}\text{ holds for }\eps,\beta,R\}.
\]
For the different cases from Theorem \ref{th_conv} we have that $T_{\eps,\beta,R}$ is well-defined and positive due to continuity. For the three cases we need to show the following:
\begin{enumerate}
    \item Let $M\geq 3$. There are $\beta,\eps_1,c>0$ such that if $\eta=\eta(\eps)\leq cR\eps^{16+M+\frac{1}{2}}$, then it holds $T_{\eps,\beta,R}=T_0$ for all $\eps\in(0,\eps_1]$.
    \item Let $M\geq 3$. The analogous statement as in the first case holds when $R$ is small and $M$ in the estimates is replaced by $2$.
    \item Let $M=2$. The analogous statement as in the first case holds for $\beta=0$ and small time $T$.
\end{enumerate}

In the following we first carry out a general computation that works for every case. Taking the difference of the nonlocal Allen-Cahn equation \eqref{eq_ACnonloc_1} for $c_\eps$ and the local Allen-Cahn equation \eqref{eq_ACloc_1} for $c^A_\eps$ with remainder $r^A_\eps$, we obtain
\begin{align}\label{eq_conv_proof_equ}
\partial_t\overline{c}_\eps -\Delta\overline{c}_\eps +f''(c^A_\eps)\overline{c}_\eps +\Lc_\eta c_\eps + \Delta c_\eps = -r^A_\eps - r_\eps(c_\eps,c^A_\eps), 
\end{align}
where we have set $r_\eps(c_\eps,c^A_\eps):=\frac{1}{\eps^2}\left[f'(c_\eps)-f'(c^A_\eps)-f''(c^A_\eps)\overline{c}_\eps\right]$. Testing \eqref{eq_conv_proof_equ} with $g_\beta^2\overline{c}_\eps$ and integrating over $\T^n\times(0,T)$ for $T\in(0,T_{\eps,\beta,R,M}]$ yields
\begin{align*}
&\frac{1}{2}g_\beta^2(T)^2\|\overline{c}_\eps(T)\|_{L^2(\T^n)}^2 
-\frac{1}{2}\|\overline{c}_\eps(0)\|_{L^2(\T^n)}^2 
+\beta\int_0^T g_\beta^2\|\overline{c}_\eps\|_{L^2(\T^n)}^2\,\text{d}t\\
&+\int_0^T g_\beta^2 \int_{\T^n}|\nabla\overline{c}_\eps|^2 + \frac{1}{\eps^2} f''(c^A_\eps)|\overline{c}_\eps|^2\,\text{d}x\,\text{d}t+\int_0^T g_\beta^2\int_{\T^n}(\Lc_\eta c_\eps+\Delta c_\eps)\overline{c}_\eps\,\text{d}x\,\text{d}t\\
&=-\int_0^T g_\beta^2 \int_{\T^n} r^A_\eps\overline{c}_\eps+r_\eps(c_\eps,c^A_\eps)\overline{c}_\eps\,\text{d}x\,\text{d}t,
\end{align*}
where we used $\frac{d}{dt}g_\beta=-\beta g_\beta$. We have $\frac{1}{2}\|\overline{c}_\eps(0)\|_{L^2(\T^n)}^2\leq\frac{1}{2}R^2\eps^{2M+1}$ due to the assumption in the theorem. Moreover, the spectral estimate for the local case from Theorem \ref{th_local_spectral} yields
\begin{align*}
&\int_0^T g_\beta^2 \int_{\T^n}|\nabla\overline{c}_\eps|^2 + \frac{1}{\eps^2} f''(c^A_\eps)|\overline{c}_\eps|^2\,\text{d}x\,\text{d}t \\
&\geq -\overline{C}\int_0^T g_\beta^2\|\overline{c}_\eps\|_{L^2(\T^n)}^2\,\text{d}t
+\int_0^T g_\beta^2 \left(\|\nabla\overline{c}_\eps\|_{L^2(\T^n\setminus\Gamma_t(\delta))}^2+\|\nabla_\tau\overline{c}_\eps\|_{L^2(\Gamma_t(\delta))}^2\right)\text{d}t.
\end{align*}
Furthermore, we use the estimate $\|\Lc_\eta c_\eps + \Delta c_\eps\|_{L^2(\T^n)}\leq C\eta \|c_\eps\|_{H^3(\T^n)}$ from Theorem \ref{th_AbelsHurm} and the uniform $H^3$-estimate for $c_\eps$ from Theorem \ref{th_H^3_estimate} for $0<\eta\leq c\eps^4$. This yields
\begin{align*}
&\left|\int_0^T g_\beta^2\int_{\T^n}(\Lc_\eta c_\eps+\Delta c_\eps)\overline{c}_\eps\,\text{d}x\,\text{d}t\right|
\leq \int_0^T g_\beta^2 C\eta\eps^{-16}\|\overline{c}_\eps\|_{L^2(\T^n)}\,\text{d}t\\
&\leq\frac{\beta}{2}\int_0^T g_\beta^2\|\overline{c}_\eps\|_{L^2(\T^n)}^2\,\text{d}t 
+C\|g_\beta\|_{L^2(0,T)}^2\,(\eta\eps^{-16})^2.
\end{align*}
Moreover, using the remainder estimate for $r^A_\eps$ from Theorem \ref{th_local_approx} and an integral transformation in tubular neighbourhood coordinates yields
\[
\left|\int_0^T g_\beta^2\int_{\T^n} r^A_\eps\overline{c}_\eps\,\text{d}x\,\text{d}t\right| 
\leq C\int_0^T g_\beta^2 \eps^{M+\frac{1}{2}} \|\overline{c}_\eps\|_{L^2(\T^n)} \leq \overline{C}_1 \|g_\beta\|_{L^1(0,T)} R\eps^{2M+1}, 
\]
where we used \eqref{eq_conv_proof_estimate}. Additionally, we use the uniform boundedness of $c_\eps$ and $c^A_\eps$ from Theorem \ref{th_existence} and Theorem \ref{th_local_approx} to estimate with Taylor
\[
\left|\int_0^T g_\beta^2 r_\eps(c_\eps,c^A_\eps)\overline{c}_\eps\,\text{d}t\right|\leq \frac{C}{\eps^2}\int_0^T g_\beta^2 \|\overline{c}_\eps\|_{L^3(\T^n)}^3\,\text{d}t.
\]
The latter can be estimated by splitting $\T^n$ into $\T^n\setminus\Gamma(\delta)$ and $\Gamma(\delta)$, using tubular neighbourhood coordinates for the latter set as well as Gagliardo-Nirenberg estimates, cf.~Moser \cite{MoserACvACND}, Lemma 6.6. This implies
\[
\left|\int_0^T g_\beta^2 r_\eps(c_\eps,c^A_\eps)\overline{c}_\eps\,\text{d}t\right|
\leq C R^3 \eps^{2M+1} \eps^{M-2} \|g_\beta^{-1}\|_{L^\frac{4}{4-n}(0,T)}.
\]
Finally, we control the $\partial_n\overline{c}_\eps$-term in \eqref{eq_conv_proof_estimate} by the spectral term via
\begin{align*}
    \eps^2\|g_\beta\partial_n\overline{c}_\eps\|_{L^2(\T^n\times(0,T)\cap\Gamma(\delta))} ^2 
    \leq  C&\eps^2 \int_0^T g_\beta^2 \int_{\T^n}|\nabla\overline{c}_\eps|^2 + \frac{1}{\eps^2} f''(c^A_\eps)\overline{c}_\eps ^2\,\text{d}x\,\text{d}t\\
    &+ C\int_0^T g_\beta^2 \|\overline{c}_\eps\|_{L^2(\T^n)}^2\,\text{d}t,
\end{align*}
where we used that $c^A_\eps$ is uniformly bounded with respect to small $\eps$. The first term on the right hand side is absorbed by half of the spectral term for $\eps$ small. Altogether this yields
\begin{align}\begin{split}
&\frac{1}{2}g_\beta(T)^2\|\overline{c}_\eps(T)\|_{L^2(\T^n)}^2 
+ \frac{1}{2}\|g_\beta\nabla\overline{c}_\eps\|_{L^2(\T^n\times(0,T)\setminus\Gamma(\delta))}^2 \\
&+\frac{1}{2} \|g_\beta\nabla_\tau\overline{c}_\eps\|_{L^2(\T^n\times(0,T)\cap\Gamma(\delta))}^2 
+\frac{\eps^2}{2} \|g_\beta\partial_n \overline{c}_\eps\|_{L^2(\T^n\times(0,T)\cap\Gamma(\delta))}^2 \\
&\leq \frac{R^2}{2}\eps^{2M+1} 
+\int_0^T g_\beta^2(-\frac{\beta}{2}+\overline{C}_0)\|\overline{c}_\eps\|_{L^2(\T^n)}^2\,\text{d}t
+C\|g_\beta\|_{L^2(0,T)}^2(\eta\eps^{-16})^2\\
&+\overline{C}_1 \|g_\beta\|_{L^1(0,T)} R\eps^{2M+1}
+C R^3 \eps^{2M+1} \eps^{M-2} \|g_\beta^{-1}\|_{L^\frac{4}{4-n}(0,T)} \label{eq_conv_proof_estimate2}\end{split}
\end{align}
for all $T\in(0,T_{\eps,\beta,R}]$ and $\eps\in(0,\tilde{\eps}_0]$, where $\tilde{\eps}_0>0$ is small (independent of $\beta, R, T$).

Now we consider the distinct cases in the theorem.\\
\newline
\textit{Ad 1.} Let $M\geq 3$. We choose $\beta\geq 2\overline{C}_0$ and large such that $\frac{\overline{C}_1}{\beta}\leq\frac{R}{8}$. We now roughly estimate $\|g_\beta\|_{L^2(0,T)}^2\leq T_0^2$ for this case. Then if $\eta\leq cR\eps^{16+M+\frac{1}{2}}$ for some $c>0$ small, then the right hand side in \eqref{eq_conv_proof_estimate2} is estimated by $\frac{3}{4}R^2\eps^{2M+1}$ for all $T\in(0,T_{\eps,\beta,R}]$ and $\eps\in(0,\eps_1]$, where $\eps_1>0$ is small (depending on $R,\beta$). Finally, a contradiction and continuity argument shows $T_{\eps,\beta,R}=T_0$ for all $\eps\in(0,\eps_1]$.\\
\newline
\textit{Ad 2.} Let $M=2$. Then we set $\beta=0$ and obtain for $\eta\leq cR\eps^{16+M+\frac{1}{2}}$ with \eqref{eq_conv_proof_estimate} that the right hand side in \eqref{eq_conv_proof_estimate2} is estimated by
\[
\left[\frac{R^2}{2} + \overline{C}_0 T + CR^2T^2+\overline{C}_1 TR+CR^3 T^{\frac{4-n}{4}}\right] \eps^{2M+1}.
\]
Therefore one can argue analogously to the other cases by taking time $T\leq T_1$ and $T_1$ small uniformly with respect to $\eps$.\\

This shows Theorem \ref{th_conv}.\hfill$\square$\\
\newline
\textit{Acknowledgments.} C.~Hurm was partially supported  by the Graduiertenkolleg 2339 \textit{IntComSin} of the Deutsche Forschungsgemeinschaft  (DFG, German Research Foundation) -- Project-ID 321821685. M.~Moser has received funding from the European Research Council (ERC) under the European Union’s Horizon 2020 research and innovation programme (grant agreement No 948819).

\setcounter{secnumdepth}{0}

\makeatletter
\renewenvironment{thebibliography}[1]
{\section{\bibname}
	\@mkboth{\MakeUppercase\bibname}{\MakeUppercase\bibname}%
	\list{\@biblabel{\@arabic\c@enumiv}}%
	{\settowidth\labelwidth{\@biblabel{#1}}%
		\leftmargin\labelwidth
		\advance\leftmargin\labelsep
		\@openbib@code
		\usecounter{enumiv}%
		\let\p@enumiv\@empty
		\renewcommand\theenumiv{\@arabic\c@enumiv}}%
	\sloppy
	\clubpenalty4000
	\@clubpenalty \clubpenalty
	\widowpenalty4000%
	\sfcode`\.\@m}
{\def\@noitemerr
	{\@latex@warning{Empty `thebibliography' environment}}%
	\endlist}
\makeatother

\footnotesize

\bibliographystyle{siam}

~\\
\textit{(H.~Abels) Fakultät für Mathematik, Universität Regensburg, Universitätsstraße 31, D-93053 Regensburg, Germany \\ 
E-mail address:} \textsf{helmut.abels@mathematik.uni-regensburg.de}\\
\newline
\textit{(C.~Hurm) Fakultät für Mathematik, Universität Regensburg, Universitätsstraße 31, D-93053 Regensburg, Germany \\ 
E-mail address:} \textsf{christoph.hurm@mathematik.uni-regensburg.de}\\
\newline
\textit{(M.~Moser) Institute of Science and Technology Austria, Am Campus 1, AT-3400 Klosterneuburg\\
E-mail address:} \textsf{maximilian.moser@ist.ac.at}

\end{document}